\documentclass{amsart}
\usepackage{amsmath}
\usepackage{amsfonts}
\usepackage{amssymb}
\usepackage{graphicx,xypic,tikz,asymptote}
\usepackage{hyperref}
\providecommand{\U}[1]{\protect\rule{.1in}{.1in}}
\newtheorem{theorem}{Theorem}
\newtheorem{acknowledgement}[theorem]{Acknowledgment}

\newtheorem{condition}[theorem]{Condition}

\newtheorem{corollary}[theorem]{Corollary}

\newtheorem{definition}[theorem]{Definition}

\newtheorem{lemma}[theorem]{Lemma}

\newtheorem{proposition}[theorem]{Proposition}
\newtheorem{remark}[theorem]{Remark}

\newcommand{\R}{\mathbb{R}} \newcommand{\C}{\mathbb{C}}
\newcommand{\D}{\mathbb{D}} \newcommand{\N}{\mathbb{N}}
\newcommand{\Z}{\mathbb{Z}} 
\DeclareMathOperator{\Rm}{Rm}
\DeclareMathOperator{\Area}{Area}
\DeclareMathOperator{\radius}{radius}
\DeclareMathOperator{\Vol}{Vol}
\renewcommand{\tilde}{\widetilde} \renewcommand{\bar}{\overline}
\author{David Glickenstein} \address{University of Arizona\\617 N Santa Rita\\Tucson, AZ 85750 USA}\email{dglicken@arizona.edu}
\author{Lee Sidbury}
\title{Convergence of discrete conformal mappings on surfaces}

\begin{document}
\begin{abstract}
	Discrete conformal mappings based on circle packing, vertex scaling, and related structures has had significant activity since Thurston proposed circle packing as a way to approximate conformal maps in the 1980s. The first convergence result of Rodin-Sullivan (1987) proved that circle packing maps do indeed converge to conformal maps to the disk. Recent results have shown convergence of maps of other discrete conformal structures to conformal maps as well. We give a general theorem of convergence of discrete conformal mappings between surfaces that allows for a variety of discrete conformal structures and manifolds with or without boundary. The mappings are a composition of piecewise linear discrete conformal mappings and Riemannian barycentric coordinates, called barycentric discrete conformal maps. Estimates of the barycentric discrete conformal maps allow extraction of convergent subsequences and estimates for the pullback of the Riemannian metric, proving conformality. 
\end{abstract}
	
	\maketitle

\section{Introduction}

Conformal mappings are approximated by mappings of circle packings, often called discrete conformal mappings, as conjectured by Thurston \cite{thurston_talk} and proven later by Rodin and Sullivan \cite{rodin-sullivan}. The approximation is by piecewise linear maps between the triangulations determined by the nerve of the circle packing. However, since not every geometric triangulation is the nerve of a circle packing, there is reason to consider similar conformal structures that may allow for different or even better approximation. One such structure is the vertex scaling discrete conformal structure introduced by Luo \cite{luo_yamabe}. Convergence of conformal mappings of tori with vertex scaling
was recently proven by Luo-Sun-Wu \cite{luo-sun-wu_convergence} (see also \cite{luo-wu-zhu_convergence}).

This paper generalizes the Rodin-Sullivan Theorem in two directions. First, we consider mappings between Riemannian surfaces, not only subsets of the complex plane. 
Second, we consider the entire zoo of  discrete conformal structures described in \cite{glick-disc-laplacians} and \cite{zhang2014unified}, shown to be the
same structures \cite{glickensteinthomas}. The assumptions allow for manifolds with or without boundary. 
The approach of this paper is to consider pullbacks
of Riemannian metrics and use the definition of conformality based on such maps. While different from the approach of Rodin-Sullivan, estimates for pullbacks of discrete Euclidean structures based on circle packings is implicit in the work of Rodin \cite{rodinschwarz,rodinschwarz2}. These are applied to Riemannian manifolds through the use of Riemannian barycentric maps.
\begin{figure}\label{fig:omega}
\includegraphics[width=.7\textwidth]{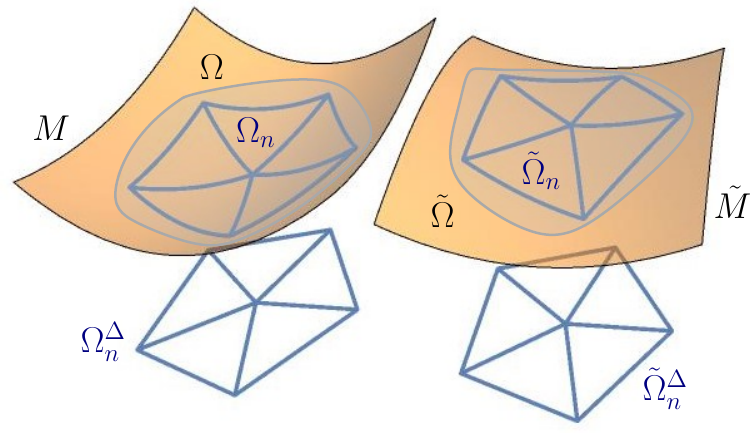}
\caption{Smooth and piecewise flat manifolds}
\end{figure}

Let $(M,g)$ and $(\tilde{M},\tilde{g})$ be complete Riemannian surfaces. In what follows, we consider pairs of submanifolds $\Omega \subseteq M$ and $\tilde{\Omega} \subseteq \tilde{M}$ and look to approximate conformal mappings between $\Omega$ and $\tilde{\Omega}$. We take sequences $\{\Omega_n\}$ and $\{\tilde{\Omega}_n\}$ that exhaust $\Omega$ and $\tilde{\Omega}$ and have geodesic triangulations with the same (combinatorial) triangulation $T_n$. These have corresponding piecewise flat triangulated spaces $\Omega^\Delta_n$ and $\tilde{\Omega}^\Delta_n$ with the same edge lengths but with flat simplices where the geometries are determined by discrete conformal parameters $f_n$ and $\tilde{f}_n$ for a given discrete conformal structure. See Figure \ref{fig:omega} for a schematic. Under appropriate assumptions (see Definition \ref{def:admissible-sequence}) we can prove the following.

\begin{theorem}
	\label{thm:main-thm}
	Let $\{(\Omega_n,\tilde{\Omega}_n,T_n,f_n,\tilde{f}_n, \epsilon_n)\}$ be an admissible
	sequence for $(\Omega, \tilde{\Omega})$ and let $\{\Phi_n:\Omega_n \to \tilde{\Omega}_n\}$ be the
	corresponding sequence of barycentric discrete conformal maps. Then the family
	$\{\Phi_n\}$ has a subsequence that converges uniformly on compact subsets of
	$\Omega$. Furthermore, if the admissible sequence is proper, then there exists
	a positive continuous function $e^F$ such that $\Phi_n^{*}\tilde{g}\to e^Fg$ in $L^{\infty}$ on compact
	subsets of $\Omega$, and hence convergence is to a conformal map.
\end{theorem}

We choose to present the result using $L^\infty$ convergence on the piecewise flat manifold, similar to the work in \cite{rodinschwarz}. This convergence is essentially the same as convergence with respect to uniform estimates in the interior of triangles that is called uniform convergence in \cite{bucking15}. This convergence is also closely related to the uniform convergence on vertices in \cite{gu_convergence,luo-sun-wu_convergence}.

The theorem relies on an estimate that the pullback
metric $\Phi_n^{*}\tilde{g}$ is close to the metric $e^Fg$, for a certain choice of function $F$, as shown in Proposition \ref{thm:main-prop}.
The maps $\Phi_n$ are constructed from a (piecewise linear) discrete conformal map $\phi_n:\Omega^\Delta_n \to \tilde{\Omega}^\Delta_n$.
The spaces $\Omega^\Delta_n$ and $\tilde{\Omega}^\Delta_n$ have piecewise flat metrics $g^\Delta_n$ and $\tilde{g}^\Delta_n$ determined by the triangulation $T_n$ and edge lengths. The edge lengths are determined by the conformal factors $f_n$ and $\tilde{f}_n$ through the discrete conformal structure (see Definition  \ref{def:disc_conformal}). The piecewise flat spaces are then related to the original manifolds through Riemannian barycentric maps (see Definition \ref{def:riem_barycentric}) $\Psi_n:\Omega^\Delta_n \to \Omega_n$ and  
$\tilde{\Psi}_n:\tilde{\Omega}^\Delta_n \to \tilde{\Omega}_n$. See Figure \ref{fig:Phi-n-first-time} for a commutative diagram.

\begin{figure}[h]
	\centering
	\[ \xymatrix{
		(\Omega_n,g) \ar[r]^{\Phi_n} & (\tilde{\Omega}_n,\tilde{g}) \\
		(\Omega^{\Delta}_n, g^{\Delta}_n) \ar[u]^{\Psi_n}
		\ar[r]^{\phi_n} & (\tilde{\Omega}^{\Delta}_n, \tilde{g}^{\Delta}_n)
		\ar[u]^{\tilde{\Psi}_n}
	}
	\]
	\caption{Definition of $\Phi_n$}
	\label{fig:Phi-n-first-time}
\end{figure}

We estimate the pullback metric $\Phi_n^{*}\tilde{g}$ by first estimating pullback
metrics under the component functions $\Psi_n^{-1}$, $\phi_n$, and
$\tilde{\Psi}_n$ and then combining these estimates into our final statement.
The Riemannian barycentric maps $\Psi_n^{-1}$ and $\tilde{\Psi}_n$ are estimated in \cite{barycentric}. To estimate the pullback metric
across the discrete conformal map $\phi_n$, we estimate edge lengths in the
image in terms of edge lengths in the domain. This result does not depend on
which particular discrete conformal structure we are working with since the
proof uses the general framework set out in Section \ref{sec:discreteconformal map}.

The conditions on an admissible sequence include a rigidity assumption, Local Discrete Conformal Rigidity (LDCR), described in Section \ref{sec:LDCR}. Also important is to rule out degeneracy of the triangulations in the sequence, which is acheived here using uniform estimates on the fullness of the triangulation as described in Section \ref{sec:bary-riemm}.


\section{Discrete conformal maps}
\label{sec:discreteconformal map}

We define the general class of discrete conformal mappings, defined as certain piecewise linear (PL) maps, that 
generalize circle packings and vertex scalings, as described in \cite{zhang2014unified,glickensteinthomas} following previous work (e.g., \cite{glick_angle_variations}). This section mainly takes definitions from \cite{glickensteinthomas}. 

First we remind the reader of the definition of piecewise flat (sometimes called piecewise linear, PL, or piecewise Euclidean if trying to emphasize the difference with piecewise hyperbolic or spherical) manifolds.
\begin{definition}
	A Riemannian manifold $(M,g)$ is said to be \emph{piecewise flat} if $g$ is a piecewise flat Riemannian metric. It is a \emph{triangulated} piecewise flat manifold if it has a triangulation $T$ of $M$ such that the $(M,g)$ is isometric to the gluing of flat simplices according to the triangulation. 
\end{definition}
The geometry of the flat simplices is typically specified by edge lengths, which entirely determine the simplex in Euclidean space up to Euclidean transformation. For this reason, we will often specify the manifold as $(M,T,\ell)$, where $T$ is a triangulation and $\ell$ is the collection of edge lengths. The specifics of gluing simplices to form a manifold is detailed in \cite{thurston_book}. Note that not every collection of edge lengths leads to a manifold, since simplices with given edge lengths can be degenerate, for instance due to a triangle inequality violation. Nondegeneracy of simplices with given edge lengths can be determined using the Cayley-Menger determinant (e.g., \cite{stein_volume_of_simplex}). The reference \cite{cheeger-muller-schrader} gives many properties of piecewise flat manifolds and their curvatures. We will often annotate a manifold and its Riemannian metric with a superscript $\Delta$, when appropriate, to indicate that the manifold is piecewise flat, such as $M^\Delta$ and $g^\Delta$. We also note that piecewise flat surfaces always admit a Delaunay triangulation \cite{bobenko-springborn}.

We can now define discrete conformal structure. For a set $S$, we use $S^*$ to denote the set of functions $S\to \R$, which can be considered a topological space as $\R^{|S|}$. $V=V(T)$ and $E=E(T)$ denote the vertices and edges, respectively, of the triangulation $T$.

\begin{definition}\label{def:disc_conformal}
	Let $(M,T)$ be a triangulated manifold and fix $\alpha\in \R^{|V|}, \eta\in
	\R^{|E|}$. The \emph{discrete conformal structure with Euclidean geometry
		 corresponding to $\alpha,\eta$} is the map $\mathcal{C}_{\alpha,\eta}:
	U\subset V^{*}\to E^{*}$ given by $f\mapsto \ell(f)$, where for each
	edge $[v_i,v_j]$, the edge length $\ell_{ij}$ satisfies
	\begin{align*}
		\ell_{ij}^2 &= \alpha_ie^{2f_i}+\alpha_je^{2f_j}+2\eta_{ij}e^{f_i+f_j}.
	\end{align*}
	The domain $U$ of this map consists of all functions $f\in V^{*}$ such that
	$(M,T,\ell)$ is a piecewise flat manifold. That is, every
	simplex is nondegenerate when the lengths of edges $e_{ij}$ are given by
	$\ell_{ij}(f)$. If $f\in U$, then $f$ is called a \emph{(discrete) conformal factor}. We denote a piecewise flat manifold determined by a discrete conformal structure as $(M,T,\ell(f))$ to indicate that the lengths $\ell$ are determined by the conformal factors $f$.
	
	Two piecewise flat manifolds $(M,T,\ell(f))$ and $(\tilde{M},T,\ell(\tilde{f}))$ with the same triangulation and same conformal structure but possibly different conformal factors are said to be \emph{discrete conformal} to each other. 
\end{definition}

\begin{remark}
	Bistellar (or Pachner) moves such as edge flips do not change the piecewise flat structure but do change the triangulation. In our definition of discrete conformal, we restrict by fixing the triangulation. However, a more general version that allows edge flips, typically to preserve the Delaunay or weighted Delaunay condition, is considered in \cite{bobenko-pinkall-springborn, gu-luo-sun-wu_uniformization}
\end{remark}

The discrete conformal structure will usually be fixed throughout and so will not be specified explicitly. Note that there are similar definitions for hyperbolic and spherical backgrounds. While there are many different discrete conformal structures described in \cite{glick-disc-laplacians}, \cite{zhang2014unified}, \cite{bobenko-lutz}, and others, it was shown in \cite{glickensteinthomas} that each of these can be described as $\mathcal{C}_{\alpha,\eta}$ for some choices of $\alpha$ and $\eta$. 

Given choices of $\alpha$ and $\eta$, the domain $U$ can be empty. This happens, for example, in
the Euclidean background when $\alpha$ and $\eta$ are negative for every vertex
and edge respectively. We always assume that $\alpha$ and $\eta$
are chosen such that the corresponding $U$ contains a nonempty open subset of $V^{*}$.

%

Two of the most important examples of discrete conformal structures are circle packing and vertex scaling.

The circle packing discrete conformal structure has edge lengths defined as sums
of radii of neighboring tangent circles. That is, the length $\ell_{ij}$ of the edge $e_{ij}$
connecting vertices $v_i$ and $v_j$ is $\ell_{ij}=r_i+r_j$, where $r_i,r_j$ are the
radii of circles with centers $v_i,v_j$ respectively.

If we define $f_i:= \log r_i$, then
\begin{align*}
	\ell_{ij}^2
	&=(r_i+r_j)^2 \\
	&= (e^{f_i}+e^{f_j})^2 \\
	&= e^{2f_i}+e^{2f_j}+2e^{f_i+f_j} \\
	&= \alpha_ie^{2f_i}+\alpha_je^{2f_j}+2\eta_{ij}e^{f_i+f_j}
\end{align*}
where $\alpha_i=1$ for every vertex $v_i$ and $\eta_{ij}=1$ for every edge
$e_{ij}$. 
The discrete conformal structure, which could be denoted $C_{1,1}$, is the \emph{circle packing discrete
conformal structure}. The circle packing discrete conformal structure is well-studied in the 
plane, on surfaces, and in dimension 3, for instance \cite{Thurs, circle_packing, cdv, chow-luo, cooper-rivin, glick-yamabe-3d, he-xu}. There are also natural generalizations to 
circles with fixed intersection angle (e.g., \cite{Thurs, marden-rodin, chow-luo}) and fixed inversive distance (e.g., \cite{bowers-stephenson, guo_inversive, luo_rigidity3, XU2018rigidityrevisited, xu_bowers_stephenson, bobenko-lutz, luo2025rigidityinfiniteinversivedistance}).

The following definitions come from \cite{luo-sun-wu_convergence}, but they have
been reworked to fit better with the notation used elsewhere.
The vertex scaling discrete conformal structure has edge lengths defined as
multiples of some base edge length. That is, let $L\in E^{*}$ be a set of
edge lengths such that $(M,T,L)$ is a piecewise flat manifold. A set of edge
lengths $\ell\in E^{*}$ is said to be \emph{related (to $L$) by a vertex
	scaling} if there is some $w\in V^{*}$ such that $\ell_{ij}=
L_{ij}e^{w_i+w_j}$ for every edge $e_{ij}$. In this case, we write $\ell=w*L$.
We see that \[\ell_{ij}^2 = L_{ij}^2e^{2w_i+2w_j},\] so if we define
$f:= 2w$ and let $\alpha_i=0$ for every vertex $v_i$ and $\eta_{ij}= L_{ij}^2/2$
for every edge, then vertex scaling is the discrete conformal structure that could be denoted
$C_{0,L^2/2}$.

This 
discrete conformal structure is called been called the
\emph{perpendicular bisector} discrete conformal structure in \cite{glick_angle_variations} since the duals to
edges are exactly the perpendicular bisectors of those edges.
For more discussion of vertex scaling discrete conformal structures see
\cite{luo_yamabe,  wu-gu-sun_rigidity, gu_convergence,luo-sun-wu_convergence, luo-wu-zhu_convergence,wu-zhu_convergence}, among others.

We can now define a discrete conformal mapping.
\begin{definition}
	\label{def:discreteconfmap}
		Let 
	$(\Omega, T)$ and $(\tilde{\Omega}, T)$ be triangulated piecewise flat manifolds where both have the same triangulation $T$. 
	The piecewise linear mapping $\phi: (\Omega, T)\to (\tilde{\Omega}, T)$ taking each triangle in $\Omega$ to its corresponding triangle in $\tilde{\Omega}$ 
	is called a \emph{discrete conformal map}.
\end{definition}

\begin{remark}
	While it is most common to consider piecewise linear maps, one can also consider projective maps to extend maps at vertices to maps within triangles. This idea is implicit in \cite{he-schramm96} for circle packing and explicitly described in \cite{bobenko-pinkall-springborn}.
\end{remark}

\section{Convergence of domain and range via Riemannian barycentric coordinates}
\label{chap:barycentric}

Riemannian barycentric coordinates will be used to show that domains and ranges converge. The discussion in this section closely follows \cite{barycentric}.

\subsection{Parameterizing Euclidean Simplices}
\label{sec:bary-eucl}

Recall the standard simplex $\Delta$ and the unit simplex $D$.
\begin{definition}
	\label{def:std-unit-simplices}
	Let $\{e_0,\dots, e_k\}$ be the standard basis of $\R^{k+1}$.
	The \emph{standard simplex $\Delta\subset \R^{k+1}$} is the convex hull of the
	points $\{e_0,\dots, e_k\}$. This simplex can also be written as
	\[\Delta= \left\{ \lambda\in \R^{k+1}\,:\,
	\sum_{i=0}^k\lambda^i=1\,\mathrm{and}\, \lambda^i\geq 0\,\text{for every $i$}
	\right\}. \]
	
	Let $\{E_1,\dots, E_k\}$ be the standard basis of $\R^{k}$.
	The \emph{unit simplex $D \subset \R^k$} is the convex hull of the points $\{0,E_1,\dots, E_k\}$.
\end{definition}

Given a collection of points $p_0, \ldots, p_n \in \R^N$, the (possibly degenerate) Euclidean $n$-simplex
$\sigma$ given by their convex hull can be parameterized using either the
standard simplex $\Delta$ or the unit simplex $D$. The two parameterizations are
the following:
\begin{equation}
	\label{eq:eucl-delta-param}
	x: \Delta\to \sigma, \qquad x(\lambda) = \sum_{i=0}^n\lambda^ip_i,
\end{equation}
\begin{equation}
	\label{eq:eucl-D-param}
	y: D\to \sigma, \qquad y(u) = Au+p_0,
\end{equation}
where $A$ is the matrix whose $i$th column is $p_i-p_0$.
Notice that $x$ is the usual barycentric map.

Let $g$ be the Euclidean metric on $\R^m$. Since the simplex $\sigma$ is a
subset of $\R^m$, the restriction $g|_{\sigma}$ is a metric on $\sigma$. We can
pull this metric back to either $\Delta$ or $D$ using $x$ or $y$ respectively.

We first consider the pullback metric $x^{*}g$ on $\Delta$. Let $v,w\in
T_{\lambda}\Delta$ for some $\lambda\in \Delta$. Then $x^{*}g$ is given by
\[ \left\langle v,w \right\rangle_{x^{*}g} = \left\langle \sum_{i=0}^n v^ip_i,
\sum_{j=0}^n w^jp_j \right\rangle_{\R^m}.\]
Any tangent vector $v\in T_{\lambda}\Delta$ lies in the hyperplane
$\sum_{i=0}^nx^i=0$, so both $v$ and $w$ are such that $\sum v^i=\sum w^i=0$. We
can therefore rewrite $\left\langle v,w \right\rangle_{x^{*}g}$
as
\begin{equation}
	\label{eq:metric-bary}
	\left\langle v,w \right\rangle_{x^{*}g} = -\frac{1}{2}\sum_{i,j=0}^n
	|p_i-p_j|^2_{\R^m}v^iw^j,
\end{equation}
since $|p_i-p_j|=\ell_{ij}$, the length of
the edge $e_{ij}$ connecting the vertices $p_i$ and $p_j$.

Next we use $y$ to pull $g$ back to $D$. We have:
\begin{align}
	\left\langle v,w \right\rangle_{y^{*}g}
	&= \left\langle Av,Aw \right\rangle_g \\
	&= \sum_{i,j=1}^n \left\langle p_i-p_0, p_j-p_0 \right\rangle_{\R^m}v^iw^j.
\end{align}
So this metric can be written as a matrix
	\[  
\begin{pmatrix}
	\ell_{01}^2 & -\ell_{01}\ell_{01}\cos \phi_{12} \\[.75em]
	-\ell_{01}\ell_{02}\cos \phi_{12} & \ell_{02}^2
\end{pmatrix},\]
	where $\phi_{12}$ is the angle between $p_1-p_0$ and $p_2-p_0$ in triangle $\{p_0,p_1,p_2\}$.


Parameterizing by $\Delta$ gives us a symmetric formulation where none of the
vertices are privileged, but it has the downside that the metric $x^{*}g$ is not
determined in the direction perpendicular to the tangent plane
$T_{\lambda}\Delta$. On the other hand, parameterizing by $D$ privileges the
vertex $p_0$, but the metric $y^{*}g$ is determined in all directions. Which
parameterization we use will be a matter of convenience.

\subsection{Parameterizing Riemannian Simplices}
\label{sec:bary-riemm}

Next we parameterize Riemannian simplices using $\Delta$. In order to
see how to do this, let us examine Euclidean barycentric coordinates a little
more carefully.
The Euclidean barycentric map $x: \Delta\to \sigma$ is defined by
$x(\lambda)=\sum_i \lambda^ip_i$. This definition does not generalize directly to Riemannian manifolds, so we
need an alternative.

Consider the function
\[ E_{\lambda}(a) := \sum_{i=0}^n\lambda^i|a-p_i|^2_{\R^m}. \] We can calculate
partial derivatives of this function as follows:
\begin{align*}
	\frac{\partial E_{\lambda}}{\partial a^j}
	&= \sum_{i=0}^n \lambda^i \frac{\partial}{\partial a_j} \left( |a-p_i|^2_{\R^m} \right) \\
	&= 2\sum_{i=0}^n \lambda^i (a^j-p_i^j) \\
	&= 2 \left( a^j -x^j(\lambda) \right).
\end{align*}
The last equality follows from the fact that $\sum_i\lambda^i=1$ for any
$\lambda\in \Delta$. Hence $a=x(\lambda)$ is a critical point of $E_{\lambda}$.
Further, the Hessian $H_{E_{\lambda}}$ is very simple, being twice the identity
matrix. Thus the Hessian is positive definite and hence $x(\lambda)$ is the
minimizer of the function $E_{\lambda}(a)$. We use a generalized version of
$E_{\lambda}(a)$ to define Riemannian barycentric coordinates.

Let $(M,g)$ be a complete $m$-dimensional Riemannian manifold with $m>1$ and let
$\Delta$ be the $n$-dimensional standard simplex \[\Delta=
\left\{ \lambda\in \R^{n+1}\, : \, \lambda^i\geq
0, \sum_i \lambda^i=1 \right\}.\] Let $p_0,\dots, p_n$ be points in $M$ and
consider the function $E:M\times \Delta\to \R $ given by
\begin{equation}
	\label{eq:E}
	E(a,\lambda) = \sum_{i=0}^n \lambda^i d_g^2(a,p_i),
\end{equation}
where $d_g$ is the Riemannian distance function on $M$. Once we are in the
Riemannian setting, the function $E(\cdot, \lambda)$ may not have a minimizer or
it may have multiple. However, the following proposition says that as long as we
restrict our attention to a small enough subset of $M$, minimizers exist and are
unique.

\begin{proposition}[Proposition 13 of \cite{barycentric}]
	\label{thm:bary-prop13}
	If the points $p_i$ lie in a ball whose radius is less than half the convexity
	radius, then $E(\cdot,\lambda)$ has a unique minimizer.
\end{proposition}


This point is called the Riemannian center of mass or Karcher mean. More information on these can be found in \cite{karcher77, buser-karcher, kendall, karcher14}.
%
%
%
%

We are interested in a sequence of triangulations
where the maximum edge length is going to zero, so we may as well assume that
every simplex in every one of our triangulations lies inside a ball small enough
to be a geodesic ball. Hence the Riemannian center of mass always exists and is unique.

We can now define the (Riemannian) barycentric mapping, as follows:

\begin{definition} \label{def:riem_barycentric}
	For a given $\lambda\in \Delta$, let $\Psi(\lambda)$ be the minimizer of
	$E(\cdot, \lambda)$. We call $\Psi$ the \emph{(Riemannian) barycentric
		mapping} with respect to vertices $p_0,\dots, p_n$. Its image in $M$ is
	called the corresponding \emph{Karcher simplex}.
\end{definition}

It is shown in \cite{karcher77} that local minimizers of $E(\cdot,\lambda)$ for
fixed $\lambda$ are zeros of the section $F:M\times \Delta\to TM$ given by
\begin{equation}\label{eq:centerofmass} 
	F(a,\lambda) = \sum_{i=0}^n \lambda^iX_i|_a, \qquad \text{where}\quad X_i(a)=
	\frac{1}{2}\mathrm{grad}_a \, d^2(a,p_i) = \exp^{-1}_{a}p_i.
\end{equation}

Notice that if $\lambda^i=0$ for some $i$, then the value of $\Psi(\lambda)$ is
independent of $p_i$. Hence each facet of $\Delta$ is mapped to a Karcher
subsimplex and this subsimplex only depends on the vertices it contains.
Furthermore, since $\Psi$ is continuous, the Karcher subsimplices form the
boundary of a Karcher simplex. That is, $\partial(\Psi(\Delta)) =
\Psi(\partial\Delta)$.

%

\begin{lemma}
	\label{thm:Psi-diffeo-closed-simplices}
	Let $\Delta$ be the (closed) standard $k$-simplex and suppose the Riemannian
	barycentric map $\Psi:\Delta\to M$ has image inside a convex subset of $M$. Then $\Psi$ is a diffeomorphism onto its image.
\end{lemma}

\begin{proof}
	Clearly $\Psi$ is a diffeomorphism on the interior of $\Delta$. All that
	remains is to show that $\Psi$ is smooth on $\partial \Delta$. This follows from its formulation as the zero set of $F$ in (\ref{eq:centerofmass}) and the implicit function theorem.
%
%
%
%
\end{proof}

The main result of \cite{barycentric} is Theorem \ref{thm:barycentric-thm2}
below, which gives estimates of the pullback metric $\Psi^{*}g$ on $\Delta$ in
terms of a different metric, $g^\Delta$, defined by
\[ \left\langle v,w \right\rangle_{g^\Delta} := - \frac{1}{2} \sum_{i,j=0}^n
d^2_g(p_i,p_j)\, v^iw^j, \] for all $v,w\in T_{\lambda}\Delta$.

Note that if we take a Euclidean $n$-simplex $\sigma$ whose edge lengths
$\ell_{ij}$ are determined by the corresponding geodesic edge length in $M$,
$\ell_{ij}:= d_g(p_i,p_j)$, then $g^\Delta$ is the pullback to $\Delta$ of the
induced Euclidean metric on $\sigma$. We frequently refer
to the induced Euclidean metric on $\sigma$ as $g^{\Delta}$. 

\begin{remark}
	Note that we consider the underlying simplex $\Delta$ as the coordinate and then use different flat metrics on $\Delta$ rather than considering the non-regular simplex as the coordinate itself. This has certain advantages as we change the metrics. We may also pull back the metric to the unit simplex $D$ when needed, for instance in Section \ref{sec:eF}.
\end{remark}

Note $g^\Delta$ is not a metric for every set of edge lengths we could choose. A
necessary and sufficient condition for $g^\Delta$ to yield a metric is that
there is a nondegenerate Euclidean simplex $\sigma$ with edge lengths given by
$\{\ell_{ij}\}$.
In fact, we will assume our triangulations satisfy a slighter stronger
condition, namely, every Euclidean simplex $(\Delta,g^{\Delta})$ must be
$(\vartheta,\epsilon)$-full, according to the following definition.

\begin{definition}
	\label{def:fullness}
	A $n$-simplex $\sigma$ with Riemannian metric $g$ is
	\emph{$(\vartheta,\epsilon)$-full} if all edges have length less than or equal
	to $\epsilon$ and \[ n! \Vol_g(\sigma) \geq \vartheta \epsilon^n, \]
	where $\Vol_g(\sigma)$ is the Riemannian volume.
	
	If $T$ is a triangulation such that every simplex in $T$ is
	$(\vartheta,\epsilon)$-full, then we say that $T$ is a
	\emph{$(\vartheta,\epsilon)$-full triangulation}.
\end{definition}

\begin{remark}
	In Euclidean space it is sufficient in the above definition to check that the highest dimensional simplices are $(\vartheta, \epsilon)$-full. If we let $\sigma^{n-1}$ be a $(n-1)$-simplex that is part of a $(\vartheta, \epsilon)$-full $n$-simplex $\sigma^n$ then 
	\[
		n!\frac{1}{n} \epsilon \Vol(\sigma^{n-1}) \geq n! \Vol(\sigma^n) \geq \vartheta \epsilon^n
	\]
	implies
	\[
		(n-1)! \Vol(\sigma^{n-1}) \geq	\vartheta \epsilon^{n-1}.
	\]
	Note that in Theorem \ref{thm:barycentric-thm2} and Corollary \ref{thm:bary-estimates} the fullness assumption is only on the piecewise flat manifolds.
\end{remark}

In the following theorem, $n$ is the dimension of the simplex and let $C_0,C_1$
be constants such that $|\Rm|_g \leq C_0$ and $|\nabla \Rm|_g \leq C_1$ where $\Rm$ is the Riemannian curvature of the  manifold $M$. See
\cite{barycentric} for more discussion and a proof.

\begin{theorem}[Theorem 2 of \cite{barycentric}]
	\label{thm:barycentric-thm2}
	There exist constants $\alpha=\alpha(n,\vartheta,C_0,C_1)$, $\beta=
	\beta(n,\vartheta,C_0)$, and $\gamma=\gamma(n,\vartheta,C_0,C_1)$ such that if
	$\epsilon<\alpha$ and $(\Delta, g^\Delta)$ is a $(\vartheta,\epsilon)$-full
	simplex then
	\[ |(\Psi^{*}g-g^\Delta)(v,w)|\leq \beta \epsilon^2
	|v|_{g^\Delta}|w|_{g^\Delta}, \] and
	\[ |\nabla^e \Psi^{*}g(u,v,w)|\leq \gamma
	\epsilon|u|_{g^\Delta}|v|_{g^\Delta}|w|_{g^\Delta} \] for tangent vectors
	$u,v,w\in T_{\lambda}\Delta$ at any $\lambda\in \Delta$.
\end{theorem}

The following is a direct consequence of Theorem \ref{thm:barycentric-thm2}.
\begin{corollary}
	\label{thm:bary-estimates}
	Let $(M,g)$ be a Riemannian manifold and let $(\Delta, g^{\Delta})$ be a
	$(\vartheta, \epsilon)$-full simplex, where $\epsilon$ is small enough so that
	the Riemannian barycentric map $\Psi: \Delta\to M$ is smooth (see Theorem \ref{thm:barycentric-thm2}). Then there
	exists a constant $\beta=\beta(n,\vartheta,C_0)$ such that
	\begin{equation}
		\label{eq:Psi-estimate}
		(1-\beta\epsilon)\left| X \right|^2_{g^{\Delta}}\leq |X|^2_{\Psi^{*}g}
		\leq (1+\beta\epsilon) \left| X \right|^2_{g^{\Delta}}
	\end{equation}
	for $X\in T_{\lambda}\Delta$ and
	\begin{equation}
		\label{eq:Psi^-1-estimate}
		(1-\beta\epsilon)|X|^2_g\leq |X|^2_{(\Psi^{-1})^{*}g^{\Delta}}
		\leq (1+\beta\epsilon)|X|^2_g
	\end{equation}
	for $X\in T_pM$ for every $\lambda$ in $\Delta$ and every $p$ in
	$\Psi(\Delta)$.
\end{corollary}

\begin{proof}
	Let $\lambda$ be a point in $\Delta$ and let $X\in
	T_{\lambda}\Delta$. Then by Theorem \ref{thm:barycentric-thm2},
	\[ \left| \left( \Psi^{*}g - g^{\Delta} \right)(X,X) \right| \leq
	\beta\epsilon^2 |X|^2_{g^{\Delta}}. \] We can rewrite this inequality as
	\begin{equation}
		\label{eq:bary-estimate-middle-step}
		(1-\beta\epsilon^2)|X|^2_{g^{\Delta}}
		\leq |X|^2_{\Psi^{*}g}
		\leq (1+\beta\epsilon^2)|X|^2_{g^{\Delta}}.
	\end{equation}
	which immediately gives us \eqref{eq:Psi-estimate} since $1-\beta\epsilon\leq
	1-\beta\epsilon^2$ and $1+\beta\epsilon^2\leq 1+\beta\epsilon$ for small
	$\epsilon$. The other inequality is similar.
%
%
%
%
%
\end{proof}

\subsection{Convergence of domain and range}
In this section we set the framework for considering convergence of the domain and range of the discrete conformal mappings. Recall that each Riemannian simplex has a corresponding Euclidean simplex determined by the same edge lengths, as described in Section \ref{sec:bary-riemm}.

\begin{definition} \label{def:exhaustion}
	Let $(M,g)$ be a complete Riemannian manifold, let $\Omega\subset M$ be an
	embedded submanifold, and let $\epsilon_n$ be a sequence of positive real numbers decreasing to zero. Let $\Omega_n\subset M$ be a
	sequence of compact submanifolds with geodesic triangulations, and let $\Omega_n^\Delta$ denote the corresponding piecewise flat manifold constructed by gluing Euclidean simplices determined by the same edge lengths. 
	
	Suppose the simplices in $\Omega_n$ are small enough such that there exist Riemannian barycentric maps  $\Psi_n: \Omega_n^\Delta \to \Omega_n$. Further, suppose for each compact $K\subset \Omega$ and each open $U \subset M$ containing $\Omega$: there exists $N=N(K,U)>0$ and $\vartheta = \vartheta(K)$ such that for every $n>N$,
	\begin{enumerate}
		\item $K\subset \Omega_n$,
		\item $\Omega_n\subset U$, and
		\item for each simplex in $\Omega_n$
		that has at least one vertex in $K$, the corresponding simplex in $\Omega_n^\Delta$ is
		$(\vartheta,\epsilon_n)$-full.
	\end{enumerate}
	In this case, we say that the sequence of triples $\{(\Omega_n,T_n,\epsilon_n)\}$ is a
	\emph{generalized triangulated exhaustion of $\Omega$}.
\end{definition}


\begin{remark}
	It
	would be slightly easier to simply take $\{\Omega_n\}$ to be an exhaustion of
	$\Omega$ by compact sets and in practice this is often exactly what we do.
	However, there are certain contexts where it makes sense to have more freedom in
	our sequence of compact subsets. For example, in \cite{glick_disc_unif}, compact
	manifolds are taken that are larger than the limit manifold. Our definition
	allows for this flexibility.
\end{remark}

For each $n$, let $\Omega^\Delta_n:= \Psi_n^{-1}(\Omega_n)$ denote the
piecewise flat manifold corresponding to $\Omega_n$. We will frequently drop the
$n$ from the subscript when no confusion arises from doing so.

\begin{definition}
	Let $\Omega$ and $\tilde{\Omega}$ have generalized triangulated exhaustions
	$\{(\Omega_n, T_n, \epsilon_n)\}$ and $\{(\tilde{\Omega}_n, T_n,\epsilon_n)\}$, where both have the same triangulation $T_n$, such that there exist discrete conformal
	maps $\{\phi_n:\Omega^\Delta_n\to \tilde{\Omega}^\Delta_n\}$. For each $n$
	define $\Phi_n:\Omega_n\to\tilde{\Omega}_n$ as the composition $\Phi_n:=
	\tilde{\Psi}_n\circ\phi_n\circ \Psi_n^{-1}$, where $\Psi_n,\tilde{\Psi}_n$ are
	the Riemannian barycentric maps on
	$\Omega^\Delta_n,\tilde{\Omega}^\Delta_n$ respectively. This map $\Phi_n$
	is called a \emph{barycentric discrete conformal map}. See Figure \ref{fig:Phi-n-first-time}.
\end{definition}


\begin{lemma}
	\label{thm:Phi-n-diffeo-closed-simplices}
	The barycentric discrete conformal map $\Phi_n$ is a diffeomorphism on each
	closed simplex $\sigma$.
\end{lemma}

\begin{proof}
	By Lemma \ref{thm:Psi-diffeo-closed-simplices},
	both $\tilde{\Psi}_n$ and $\Psi_n^{-1}$ are diffeomorphisms on closed
	simplices, so all that remains is to note that $\phi_n$ is piecewise linear
	and hence clearly smooth even on the boundary of a simplex.
\end{proof}

\section{Convergence of discrete conformal factors}
\label{chap:eF}
In this section we will prove that discrete conformal factors converge to a function on the manifolds. 

\subsection{Local Discrete Conformal Rigidity Condition}
\label{sec:LDCR}

Let $\{T_n\}$ be a sequence of triangulations 
and $\{f_n\}$, $\{{\tilde{f}_n}\}$ sequences of discrete conformal factors for the the discrete conformal structure $\mathcal{C}$. We also need the following definitions about
combinatorial closed disks. In the following, we use $d_{\rm{comb}}$ to denote the combinatorial distance on the unweighted graph determined by the vertices and edges of the triangulation.

\begin{definition}
	A \emph{combinatorial closed disk} is an abstract simplicial complex that
	triangulates a topological closed disk. That is, a complex that is finite,
	connected, simply connected, and has nonempty boundary.
%
%

	A \emph{combinatorial closed disk of generation $m$} is a combinatorial
	closed disk $D_m$ with the property that there exists a vertex $v_0\in D_m$
	such that for every boundary vertex $w\in \partial D_m$, the combinatorial distance
	$d_{\rm{comb}}(v_0,w)=m$. The vertex $v_0$ is referred to as the
	\emph{center} of the combinatorial closed disk, or the disk can be said to be
	\emph{centered at} $v_0$.

When a subset of a triangulated manifold is (combinatorially) isomorphic
to a combinatorial closed disk, then we will call that subset a \emph{realized
	combinatorial closed disk}.
\end{definition}

We are now ready to define the LDCR condition, which is essential for showing convergence to conformal maps.

\begin{condition}[Local Discrete Conformal Rigidity]
	\label{def:LDCR}
	A sequence $\{(T_n,f_n,{\tilde{f}_n})\}_{n=1}^{\infty}$ is said to
	\emph{satisfy Local Discrete Conformal Rigidity (LDCR)} if there is a
	sequence of real numbers $\{s_m\}$ decreasing to zero such that for any vertex
	$v\in T_n$ that is the center of a realized combinatorial closed disk of
	generation $m$ there exists a number $\mathcal{N}_m\in\N$ such that if $n\geq \mathcal{N}_m$ and
	$w$ is a vertex adjacent to $v$ then
	\begin{equation}
		\label{eq:LDCR}
		\left| \frac{e^{f_n(v)}}{e^{f_n(w)}}\frac{e^{{\tilde{f}_n}(w)}}{e^{{\tilde{f}_n}(v)}} -1 \right| \leq s_m.
	\end{equation}
	Note that we can always take a subsequence such that we can replace $\mathcal{N}_m$ with
	$m$.

 Let  $M_{n}^\Delta= (M_n,T_n,f_n)$
and ${\tilde{M}_n}^\Delta= (\tilde{M}_n,T_n,{\tilde{f}_n})$ be piecewise flat manifolds with the
same underlying triangulation and discrete conformal structure, but possibly different
discrete conformal factors and hence different edge lengths. Then we say that
the sequence of pairs $\{(M_n^\Delta,{\tilde{M}^\Delta_n})\}$ \emph{satisfies LDCR} if
$\{(T_n,f_n,{\tilde{f}_n})\}$ satisfies $LDCR$.
\end{condition}

\begin{remark}
	It is important to note that satisfying LDCR gives the sequence $s_m$. It might be more clear to use the term \emph{$\{s_m\}$-LDCR} but we choose not to do so for brevity. Also note that the LDCR condition assumes a discrete conformal structure.
\end{remark}

We now prove an easy consequence of the LDCR condition. Define $H_n: V(T_n)\to\R$ by
$H_n(v)=e^{{\tilde{f}_n}(v)}/e^{f_n(v)}$. This function is a generalization of the
ratio of radii function in circle packing as discussed in \cite{rodin-sullivan,rodinschwarz}.

\begin{lemma}
	\label{thm:H(v)-close-to-H(w)}
	Suppose the sequences $M_n^\Delta=( M,
	T_n, f_n)$ and ${\tilde{M}_n^\Delta} = (\tilde{M}, T_n, \tilde{f}_n)$ are such that $\{(M_n^\Delta,{\tilde{M}_n^\Delta})\}$ satisfy LDCR. Suppose further that $v$ and $w$ are adjacent vertices and each is the
	center of a realized combinatorial closed disk of generation at least $m$,
	where $m$ is large enough so that $s_m\leq 1$. Then
	\begin{equation}
		\label{eq:H}
		(1-s_m)H_n(v)\leq H_n(w)\leq (1+s_m)H_n(v)
	\end{equation}
	and
	\begin{equation}
		\label{eq:H-squared}
		(1-3s_m)H_n^2(v)\leq H_n^2(w)\leq (1+3s_m)H_n^2(v).
	\end{equation}
\end{lemma}

\begin{proof}
	The first inequality is simply a rearrangement of \eqref{eq:LDCR}. For the
	second inequality, we square \eqref{eq:H} to get
	\[(1-s_m)^2H_n^2(v)\leq H_n^2(w)\leq (1+s_m)^2H_n^2(v) \] and then note that
	$1-3s_m\leq (1-s_m)^2\leq (1+s_m)^2\leq 1+3s_m$ for $s_m\leq 1$.
\end{proof}

\subsection{Combinatorial closed disks and $(\vartheta,\epsilon)$-full
	triangulations}
\label{sec:comb-closed-disks}

The next two lemmas allow us to relate combinatorial distances in realized
combinatorial closed disks to geometric distances, at least for triangulations
which are $(\vartheta,\epsilon)$-full and piecewise flat. Recall Definition \ref{def:fullness}.


\begin{lemma}
	\label{thm:fullness-min-distance}
	Let $\Delta$ be a $(\vartheta,\epsilon)$-full Euclidean simplex of dimension
	$k$. Label the vertices of $\Delta$ as $\{v_0,v_1,\dots, v_k\}$ and label the
	$(k-1)$-dimensional face opposite vertex $v_i$ by $F_{k-1}(i)$. Let $h_k$ be
	shortest distance between a vertex and the opposite face. That is,
	\[h_k = \min_i d(v_i, F_{k-1}(i)). \] Then $h_k$ satisfies
	\begin{equation}
		\label{eq:fullness-min-height}
		\vartheta \epsilon\leq h_k\leq \epsilon.
	\end{equation}
\end{lemma}

\begin{proof}
	The longest edge in $\Delta$ being no larger than $\epsilon$ implies
	immediately that the height $h_k\leq \epsilon$.
	
	For the other inequality, let $|F_{k-1}|$ denote the volume of the
	$(k-1)$-dimensional face whose distance from its opposite vertex is the
	smallest. Then the volume of the $k$-simplex $\Delta$ can be written as \[
	\Vol(\Delta) = \frac{1}{k} h_k |F_{k-1}|.\]
	Since $F_{k-1}$ is a $(k-1)$-simplex in its own right, we can write its volume
	as \[ |F_{k-1}| = \frac{1}{k-1} h_{k-1}|F_{k-2}|\] and hence
	\[ \Vol(\Delta) = \frac{1}{k(k-1)} h_kh_{k-1}|F_{k-2}|.\]
	
	Proceeding in this manner, we find that
	\[ \Vol(\Delta) = \frac{1}{k!} h_kh_{k-1}\dots h_2|F_1|, \] where
	$|F_1|$ is the length of an edge.
	
	Since $\Delta$ was assumed $(\vartheta,\epsilon)$-full, $|F_1|\leq
	\epsilon$. Further, since the heights $h_i$ are clearly no larger than the
	length of the longest side in their subsimplex, they too are such that
	$h_i\leq \epsilon$.
	
	Since $\Delta$ was assumed $(\vartheta,\epsilon)$-full, $k! \Vol(\Delta)
	\geq \vartheta\epsilon^k$, so \[ k!\Vol(\Delta) = \left(\prod_{i=2}^kh_i\right)
	|F_1| \geq \vartheta\epsilon^k,\] or, solving for $h_k$ and using that
	$h_i\leq \epsilon$ and $|F_1|\leq \epsilon$:
	\[ h_k\geq \frac{\vartheta\epsilon^k}{|F_1|\prod_{i=2}^{k-1} h_i} \geq
	\frac{\vartheta\epsilon^k}{\epsilon^{k-1}} = \vartheta\epsilon.\]
\end{proof}

\begin{lemma}
	\label{thm:geom-diam-D}
	Let $D_m$ be a realized combinatorial closed disk of generation $m$ centered
	at the vertex $v$ such that every simplex in $D_m$ is Euclidean and
	$(\vartheta,\epsilon)$-full. Then the distance
	$\delta=d_{g^{\Delta}}(v,\partial D_m)$ from $v$ to the boundary of $D_m$
	satisfies
	\begin{equation}
		\label{eq:geom-diam-D}
		\vartheta m\epsilon\leq\delta\leq m\epsilon.
	\end{equation}
\end{lemma}

\begin{proof}
	Firstly, since every edge length in $D_m$ is no greater than $\epsilon$,
	clearly $\delta$ is no larger than $\epsilon$ multiplied by the generation.
	That is, $\delta\leq m\epsilon$.
	
	We prove the left hand side of the inequality using induction.
	
	When $m=1$, $D_1$ consists of a central vertex $v$ and a collection of
	simplices each sharing the vertex $v$. In this case, the distance from $v$ to
	$\partial D_1$ is the length of a straight line connecting the vertex $v$ to
	its opposing face in one of the simplices and hence by Lemma
	\ref{thm:fullness-min-distance}, $\delta\geq \vartheta \epsilon =
	\vartheta\epsilon m$ since $m=1$.
	
	Next assume that for every disk with $m-1$ generations the distance
	$\delta_{m-1}\geq \vartheta\epsilon (m-1)$. Take $D_m$ to be a disk with $m$
	generations and let $D_{m-1}\subset D_m$ be the subcomplex of $D_m$ with the
	same center but with $m-1$ generations.
	
	Let $\gamma$ be the shortest geodesic connecting $\partial D_{m-1}$ and
	$\partial D_m$. Then there are two possibilities: either $\gamma$ passes
	through a vertex of $D_m$ or it does not. If it does, then $\gamma$ is at
	least as long as a straight line connecting a vertex of a simplex to its
	opposite face, so the length $\ell(\gamma)\geq \vartheta \epsilon$ by Lemma
	\ref{thm:fullness-min-distance}.
	
	If $\gamma$ does not pass through any vertices, then since it is the shortest
	path between $\partial D_{m-1}$ and $\partial D_m$, it must meet both
	$\partial D_{m-1}$ and $\partial D_m$ at right angles and hence $\partial
	D_{m-1}$ and $\partial D_m$ are parallel in some neighborhood around the
	endpoints of $\gamma$, and this neighborhood is determined by vertices in
	$\partial D_{m-1}$ and $\partial D_m$. Hence we can translate $\gamma$ until
	one of its endpoints is a vertex of either $\partial D_{m-1}$ or $\partial
	D_m$ and this translation will have the same length as the original. And so
	again, $\gamma$ is at least as long as a straight line connecting a vertex to
	the opposing face in a $(\vartheta,\epsilon)$-full simplex and hence here too
	$\ell(\gamma)\geq \vartheta\epsilon$.
	
	\begin{figure}[h]
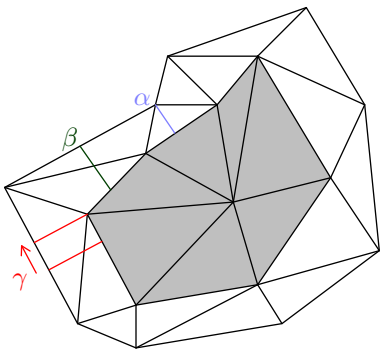

		\centering
		\begin{asy}
		size(5cm);
		defaultpen(fontsize(10pt));
		import graph;
		
		//D_(m-1)
		pair[] innerVerts =
		{(0,0),(2,0.5),(1/2,3),(-1/3,2),(-1.8,1),(-3,-.25),(-2,-2.12),(1/2,-1.7)};
		
		path m;
		for(int i=1; i<innerVerts.length; ++i) {
		m=m--innerVerts[i];
		}
		m=m--cycle;
		filldraw(m, mediumgrey);
		
		for(int i=1; i < innerVerts.length; ++i) {
		draw(innerVerts[0] -- innerVerts[i]);
		}
		
		//D_m
		pair[] outerVerts = {(-1.35,3),(-1.6,2),(-4.7,.3),(-3.2,-2.5),(-2,-3),(1,-2.5),(3,-1),(2.7,2),(1.5,4)};
		
		path M = outerVerts[0];
		for(int i=0; i<outerVerts.length; ++i) {
		M=M--outerVerts[i];
		}
		M = M -- cycle;
		draw(M);
		
		//connecting spokes
		for(int i=0; i<4; ++i) {
		draw(outerVerts[i] -- innerVerts[2+i]);
		draw(outerVerts[i] -- innerVerts[3+i]);
		}
		
		draw(outerVerts[4] -- innerVerts[6]);
		draw(outerVerts[4] -- innerVerts[7]);
		draw(outerVerts[5] -- innerVerts[7]);
		draw(outerVerts[6] -- innerVerts[7]);
		draw(outerVerts[6] -- innerVerts[1]);
		draw(outerVerts[7] -- innerVerts[1]);
		draw(outerVerts[7] -- innerVerts[2]);
		draw(outerVerts[8] -- innerVerts[2]);
		
		//now for the geodesics
		
		//first geodesic, from outerVerts[1] to line connecting innerVerts[3] and innerVerts[4]:
		
		real slope = (innerVerts[3].x-innerVerts[4].x)/(innerVerts[4].y-innerVerts[3].y);
		
		real Y(real X) {
		return outerVerts[1].y+(X-outerVerts[1].x)*slope;
		}
		path geo1 = graph(Y, outerVerts[1].x, -1.203);
		
		draw(geo1, lightblue);
		label('$\\alpha$', geo1, 2.5*NW, lightblue);
		
		//second geodesic, between angled lines:
		
		// the points
		// along outerVerts[1] -- outerVerts[2]
		real tp = .5;
		pair p = (outerVerts[1].x*(1-tp)+outerVerts[2].x*tp, outerVerts[1].y*(1-tp)+ outerVerts[2].y*tp);
		
		// along innerVerts[4] -- innerVerts[5]
		real tq = .6;
		pair q = (innerVerts[4].x*(1-tq)+innerVerts[5].x*tq, innerVerts[4].y*(1-tq)+ innerVerts[5].y*tq);
		
		// drawing the path
		draw(p -- q, darkgreen);
		label('$\\beta$', p -- q, 3*NW, darkgreen);
		
		//geodesics along the parallel:
		// inner path: innerVerts[5] -- innerVerts[6];
		// outer path: outerVerts[1] -- outerVerts[2];
		
		//arbitrary outer point
		real ta = .3;
		pair a = (innerVerts[5].x*(1-ta)+innerVerts[6].x*ta, innerVerts[5].y*(1-ta)+ innerVerts[6].y*ta);
		
		real slope = (innerVerts[5].x-innerVerts[6].x)/(innerVerts[6].y-innerVerts[5].y);
		
		real Y(real X) {
		return a.y+(X-a.x)*slope;
		}
		path geo3 = graph(Y, a.x, -3.787);
		draw(geo3, red);
		
		//geodesic from the vertex along the parallel
		real F(real X) {
		return innerVerts[5].y+(X-innerVerts[5].x)*slope;
		}
		path geo4 = graph(F, innerVerts[5].x, -4.09);
		draw(geo4, red);
		
		//sliding arrow
		real arrowSlope = -1/slope;
		
		path arrowPath = (1, arrowSlope) -- (0,0);
		
		arrowPath = scale(.3)*arrowPath; arrowPath = shift(-4.35,-.9)*arrowPath;
		
		draw(arrowPath, red, arrow=ArcArrow(SimpleHead));
		label('$\\gamma$', arrowPath, SW+S, red);
		
		\end{asy}
		\caption{Geodesics connecting $\partial D_m$ and $\partial D_{m-1}$.}
		\label{fig:geom-diam-D}
	\end{figure}
	
	Since the induction hypothesis was that $D_{m-1}$ is such that the distance
	$\delta_{m-1}\geq \vartheta \epsilon (m-1)$, we have that $\delta_m$ satisfies
	\[ \delta_m\geq \delta_{m-1}+\ell(\gamma) \geq \vartheta\epsilon (m-1)
	+\vartheta\epsilon = \vartheta\epsilon m.\]
\end{proof}

We use the next few lemmas to estimate the number of generations $m$ of a closed
combinatorial disk contained in a geodesic ball of radius $R$.

\begin{lemma}
	\label{thm:ball-radius-R}
	Let $(M,g)$ be a Riemannian manifold and let $K\subset M$ be compact. Then
	there exists $R>0$, depending only on $K$, such that for each point $p\in K$
	there is a geodesic ball of radius $R$ about $p$.
\end{lemma}

\begin{proof}
	For each point $x\in K$, there is a number $r_x\in \R_{>0}$ such that the ball
	$B_{2r_x}(x)$ of radius $2r_x$ about $x$ is a convex geodesic ball. Cover $K$ by
	balls of radius $r_x$ about each $x$. Since $K$ is compact there is a finite
	subcover $\{B_{r_i}(x_i)\}_{i=1}^N$, where $r_i:= r_{x_i}$. Let $R:=\min_ir_i$
	be the minimum over $i$ of the radii $r_i$.
	
	Now take an arbitrary point $p\in K$. By assumption there is some $i$ such
	that $p\in B_{r_i}(x_i)$. We will show that the ball $B_R(p)$ is contained in
	the convex geodesic ball $B_{2r_i}(x_i)$ and hence is itself geodesic. Let $y\in
	B_R(p)$. Then
	\begin{align*}
		d_g(y,x_i)
		& \leq d_g(y,p)+d_g(p,x_i) \\
		& < R + r_i \\
		& \leq 2r_i,
	\end{align*}
	where the last step follows since we have defined $R$ as the minimum of the
	$r_i$'s. So $d_g(y,x_i)<2r_i$ and hence $y\in B_{2r_i}(x_i)$ for every $y\in
	B_R(p)$.
	
	Thus $B_R(p)\subset B_{2r_i}(x_i)$ for some $i$ for every $p$ and hence there
	is a radius $R>0$ such that for each point $p\in K$ the ball of radius $R$
	about $p$ is a geodesic ball.
\end{proof}

\begin{lemma}
	\label{thm:R-m-epsilon}
	Let $(M,g)$ be a Riemannian manifold and let $K\subset M$ be a compact subset.
	Assume that $M$ admits a triangulation $T$ with geodesic edges such that the
	corresponding piecewise flat manifold $M^{\Delta}$ is
	$(\vartheta,\epsilon)$-full, where $\epsilon \in (0, 1/(3\beta)]$ is small
	enough so that the Riemannian barycentric map $\Psi:M^{\Delta}\to M$ is
	well-defined. (The constant $\beta=\beta(n,\vartheta,C_0)$ is given in Theorem
	\ref{thm:barycentric-thm2}.)
	
	Choose a vertex $v$ such that $v$ lies within the compact set $K$ and let
	$R=R(K)$ be chosen as in Lemma \ref{thm:ball-radius-R}. Let $D_m(v)$ be a
	realized closed combinatorial disk such that
	\begin{enumerate}
		\item $D_m(v) \subset B_R(v)$, and
		\item for any $q>m$, $D_q(v)$ has nonempty intersection with $M\setminus
		B_R(v)$.
	\end{enumerate}
	Then the following inequality holds:
	\[ m \geq \frac{R}{2\epsilon}. \]
\end{lemma}

\begin{proof}
	Since we are assuming that $D_m(v)$ is such that no larger realized
	combinatorial closed disk $D_q(v)$ is completely contained in $B_R(v)$, we
	have in particular that $D_{m+1}(v)$ has nonempty intersection with the
	complement of $B_R(v)$. Take a point $p$ in $D_{m+1}(v)$ such that $p$ is not
	in $B_R(v)$ and let $a\in M^{\Delta}$ be the image of $p$ under $\Psi^{-1}$.
	That is, $a:= \Psi^{-1}(p)$.
	
	Since $\Psi$ and $\Psi^{-1}$ are continuous everywhere and diffeomorphisms
	when restricted to (closed) simplices by Lemma
	\ref{thm:Psi-diffeo-closed-simplices}, Proposition \ref{thm:integral-estimate}
	in the Appendix tells us that \[ d_g(v,p) \leq
	\sqrt{1+\beta\epsilon}\,d_{g^{\Delta}}(v,a) \leq
	(1+\beta\epsilon)d_{g^{\Delta}}(v,a). \] Since $p$ is not in $B_R(v)$, the
	distance $d_g(v,p)$ must be no less than $R$, so \[ R\leq d_g(v,p)\leq
	d_{g^{\Delta}}(v,a)(1+\beta\epsilon), \] or, dividing through by
	$1+\beta\epsilon$,
	\begin{equation}
		\label{eq:R-m-epsilon-first-step}
		\frac{R}{1+\beta\epsilon}\leq d_{g^{\Delta}}(v,a).
	\end{equation}
	Further, by Lemma \ref{thm:geom-diam-D}, $d_{g^{\Delta}}(v,a) \leq
	\epsilon(m+1)$. If we combine these two estimates, we get that
	\[ \frac{R}{1+\beta\epsilon}\leq d_{g^{\Delta}}(v,a)\leq \epsilon(m+1), \] or,
	dropping the middle term and rearranging,
	\begin{equation}
		\label{eq:R-m-epsilon-middle-step}
		R\leq \epsilon(m+1)+\epsilon^2\beta(m+1).
	\end{equation}
	
	It is easy to check that when $\epsilon$ is small,
	\begin{equation}
		\label{eq:R-m-epsilon-last-inequality}
		\epsilon(m+1)+\epsilon^2\beta(m+1)\leq 2m\epsilon.
	\end{equation}
	Specifically, this is true when \[ \epsilon \leq \frac{m-1}{\beta(m+1)}.\]
	Finally, note that $(m-1)/\beta(m+1)$ increases with $m$, so
	\[ \frac{m-1}{\beta(m+1)}\geq \frac{2-1}{\beta(2+1)} = \frac{1}{3\beta} \] for
	all $m\geq 2$. This means that as long as $\epsilon$ is taken to be less than
	or equal to $1/(3\beta)$, \eqref{eq:R-m-epsilon-last-inequality} is satisfied
	and hence $R\leq 2m\epsilon$. Dividing through by $2\epsilon$ gives the
	result.
\end{proof}

\section{Estimates on the metric distortion of discrete conformal mappings}
\label{chap:phi_n}

We begin by showing that squared edge lengths in domain and image are relatively
close.

\begin{lemma}
	\label{thm:ell-vs-ell-bar}
	Let $\{(M_n,g^{\Delta}_n)\}$ and $\{({\tilde{M}_n}, {\tilde{g}^{\Delta}}_n)\}$ be sequences of
	piecewise flat manifolds with discrete conformal factors $f_n, {\tilde{f}_n}$
	respectively, and assume that the sequences $\{(M_n,{\tilde{M}_n})\}$ satisfy LDCR.
	 Then for each $m>0$ there exists a constant $N=N(m)$ such that for
	every edge $e_{vw}$, where the vertices $v$ and $w$ and the centers of
	realized combinatorial closed disks of generation $m$, if $n\geq N$ the following holds:
	\[ H_n^2(v)\ell_{vw}^2(1-3s_m) \leq \tilde{\ell}_{vw}^2 \leq H_n^2(v)
	\ell_{vw}^2(1+3s_m).\]
\end{lemma}

\begin{proof}
	First write
	\begin{align}
		\tilde{\ell}_{vw}^2
		&= \alpha_ve^{2{\tilde{f}_n}(v)}+\alpha_we^{2{\tilde{f}_n}(w)}+2\eta_{vw}e^{{\tilde{f}_n}(v)+{\tilde{f}_n}(w)}\nonumber \\
		&= H_n^2(v) \left( \alpha_ve^{2f_n(v)}
		+ \alpha_w e^{2f_n(w)}\frac{e^{2f_n(v)}}{e^{2{\tilde{f}_n}(v)}}\frac{e^{2{\tilde{f}_n}(w)}}{e^{2f_n(w)}} \right.	\label{eq:l_vw-bar}\\
		& \qquad \qquad \qquad \qquad \qquad \left.+ 2\eta_{vw}e^{f_n(v)+f_n(w)}\frac{e^{{\tilde{f}_n}(w)}}{e^{f_n(w)}}\frac{e^{f_n(v)}}{e^{{\tilde{f}_n}(v)}}\right). \nonumber
		\end{align}
	
	Since we assumed that the sequence of manifolds
	satisfies LDCR, we have by Lemma \ref{thm:H(v)-close-to-H(w)} that
	\begin{equation}
		\label{eq:l_vw-middle-step-1}
		1-s_m\leq \frac{e^{f_n(v)}}{e^{{\tilde{f}_n}(v)}}
		\frac{e^{{\tilde{f}_n}(w)}}{e^{f_n(w)}} \leq 1+s_m
	\end{equation}
	and
	\begin{equation}
		\label{eq:l_vw-middle-step-2}
		1-3s_m\leq \frac{e^{2f_n(v)}}{e^{2{\tilde{f}_n}(v)}}\frac{e^{2{\tilde{f}_n}(w)}}{e^{2f_n(w)}}\leq 1+3s_m.
	\end{equation}
	
	Now we use \eqref{eq:l_vw-middle-step-1} and \eqref{eq:l_vw-middle-step-2} to
	estimate the third and second terms of \eqref{eq:l_vw-bar} respectively to get
%
	\[ H_n^2(v)(1-3s_m)\ell_{vw}^2 \leq \tilde{\ell}_{vw}^2 \leq
	H_n^2(v)(1+3s_m)\ell_{vw}^2. \]
\end{proof}

For the following, for any point $a$ in a triangulated manifold, we say it is contained in a realized combinatorial closed disk about $a$ if every vertex in any simplex containing $a$ is the center of a realized combinatorial closed disk.

\begin{lemma}
	\label{thm:eF-close-to-Hv}
	Let $\{(M_n,g^{\Delta}_n)\}$ and $\{({\tilde{M}_n}, {\tilde{g}^{\Delta}}_n)\}$ be sequences of
	piecewise flat manifolds with discrete conformal factors $f_n, {\tilde{f}_n}$
	respectively such that the sequences $\{(M_n,{\tilde{M}_n})\}$ satisfy LDCR. Let $e^F_{\Delta,n}: M_n\to \R$ be defined as
	\[
		e^F_{\Delta,n}(a) =
		\sum_{i=0}^k\lambda^iH_n^2(v_i),
	\] 
	where $(\lambda^0,\dots, \lambda^k)$ are barycentric coordinates and $H_n(v)= f_n(v)/\tilde{f}_n(v)$ for any vertex $v$. (See Chapter \ref{chap:eF}.) Fix $m>0$ large enough such that $s_m\leq 1/6$. 
	Then there exists $N_m>0$ such that if $n>N_m$ then for any $a\in M_n$ contained in a realized closed combinatorial disk
	of generation $m$ about $a$ in $M_n$,
	\begin{equation}
		\label{eq:eF-bdd-Hv}
		(1-3s_m)H_n^2(v)\leq e^F_{\Delta,n}(a)\leq (1+3s_m)H_n^2(v)
	\end{equation}
	and
	\begin{equation}
		\label{eq:Hv-bdd-eF}
		(1-6s_m)e^F_{\Delta,n}(a)\leq H_n^2(v)\leq (1+6s_m)e^F_{\Delta,n}(a),
	\end{equation}
	for any vertex $v$ of a simplex containing $a$.
\end{lemma}

\begin{proof}
	Without loss of generality, choose a vertex in the simplex containing $a$ and
	label this vertex $v_0$.
	
	For the first inequality, \eqref{eq:eF-bdd-Hv}, we calculate as follows,
	making use of Lemma \ref{thm:H(v)-close-to-H(w)} for the last step.
	\begin{align*}
		\left| e^F_{\Delta,n}(a) - H_n^2(v_0) \right|
		&= \left| \sum_{i=0}^k \lambda^iH_n^2(v_i)-H_n^2(v_0) \right| \\
		&= \left| H_n^2(v_0)\sum_{i=1}^k(-\lambda^i)+\sum_{i=1}^k\lambda^iH_n^2(v_i) \right| \\
		&= \left| \sum_{i=1}^k\lambda^i(H_n^2(v_i)-H_n^2(v_0)) \right| \\
		&\leq \sum_{i=1}^k \lambda^i\left| H_n^2(v_i)-H_n^2(v_0) \right| \\
		&\leq 3 s_mH_n^2(v_0).
	\end{align*}
	
	This shows \eqref{eq:eF-bdd-Hv}. The other equation now follows by estimating $1/(1 \pm 3s_m)$ by $1 \pm 6 s_m$.
%
%
%
\end{proof}

\begin{theorem}
	\label{thm:phi-n-metric-estimate}
Let $\{(M_n,g^{\Delta}_n)\}$ and $\{({\tilde{M}_n}, {\tilde{g}^{\Delta}}_n)\}$ be sequences of
piecewise flat manifolds with discrete conformal factors $f_n, {\tilde{f}_n}$
respectively, and assume that the sequences $\{(M_n,{\tilde{M}_n})\}$ satisfy LDCR. Let $\phi_n : M_n \to \tilde{M}_n$ denote the corresponding discrete conformal mapping. Suppose furthermore that $M_n$ and $\tilde{M}_n$ are $(\vartheta,\epsilon_n)$-full and have
	dimension two. Let $e^F_{\Delta,n}$ be the linear interpolation of the weight ratios $H_n$ as in Lemma \ref{thm:eF-close-to-Hv}. 
	
	There exists a constant $C=C(\vartheta)$ such that the following is true.
	 For any $m>0$ large enough such that $s_m\leq 1/6$ there exists $N_m>0$ such that if $n>N_m$, $a\in M_n$ is contained in a realized closed combinatorial disk
	of generation $m$ about $a$ in $M_n$, and 
	$X\in T_aM_n$, then
	\[ (1-Cs_m) |X|^2_{e^F_{\Delta,n} g^{\Delta}_n} \leq
	|X|^2_{\phi_n^{*}{\tilde{g}^{\Delta}}_n} \leq (1+Cs_m)|X|^2_{e^F_{\Delta,n}
		g^{\Delta}_n}.\]
\end{theorem}

\begin{proof}
	We will be using Lemmas \ref{thm:barycentric-lemma3} and
	\ref{thm:barycentric-lemma6} from the Appendix to prove this. Lemma \ref{thm:barycentric-lemma3}
	gives us a lower bound on the eigenvalues of $e^F_{\Delta,n}g^{\Delta}$ which we
	use along with Lemma \ref{thm:H(v)-close-to-H(w)} to bound
	$|e^F_{\Delta,n}g^{\Delta}_{ij}- {\tilde{g}^{\Delta}}_{ij}|$, the metric components in the unit simplex $D$ parameterization (see Section \ref{sec:bary-eucl}). With that bound in hand,
	Lemma \ref{thm:barycentric-lemma6} gives the result.
	
	By Lemma \ref{thm:barycentric-lemma3}, we have that eigenvalues $\lambda_i$
	for $e^F_{\Delta,n}g^{\Delta}$ are all bounded below by
	\begin{align} 
		\frac{\vartheta^2\epsilon_n^2e^F_{\Delta,n}}{4} \leq \lambda_i, \label{eq:eigenvaluebound}
		\end{align} 
		since the
	eigenvalues of $e^F_{\Delta,n}g^{\Delta}$ are precisely the eigenvalues of
	$g^{\Delta}$ multiplied by $e^F_{\Delta,n}$.
	
	Next, we want to bound $\left|e^F_{\Delta,n}g^{\Delta}_{ij}-{\tilde{g}^{\Delta}}_{ij}\right|$. Note
	that by definition of these metrics, the diagonal terms $g^{\Delta}_{ii}$ and
	${\tilde{g}^{\Delta}}_{ii}$ are simply $\ell_{0i}^2$ and $\tilde{\ell}_{0i}^2$
	respectively, while the off-diagonal terms $g_{ij}^{\Delta}, i\neq j$, are
	given by
	\begin{align*}
		g_{ij}^{\Delta}
		&= \left\langle p_i-p_0, p_j-p_0 \right\rangle \\
		&= \ell_{0i}\ell_{0j}\cos(\theta_{ij}) \\
		&= \frac{1}{2}\left( \ell_{0i}^2+\ell_{0j}^2-\ell_{ij}^2 \right)
	\end{align*}
	and similarly,
	\[ \tilde{g}_{ij}^{\Delta} = \frac{1}{2} \left(
	\tilde{\ell}_{0i}^2+\tilde{\ell}_{0j}^2-\tilde{\ell}_{ij}^2 \right).\]
	
	We begin by bounding expressions of the form
	$\left|e^F_{\Delta,n}\ell_{ij}^2-\tilde{\ell}_{ij}^2\right|$. This bound
	immediately gives us bounds on the diagonal terms $\left|
	e^F_{\Delta,n}g_{ii}^{\Delta}-\tilde{g}_{ii}^{\Delta} \right|$ and we use the triangle
	inequality to bound the off-diagonal terms $\left|
	e^F_{\Delta,n}g_{ij}^{\Delta}-\tilde{g}_{ij}^{\Delta} \right|$ for $i\neq j$.
	
	We start by combining Lemmas \ref{thm:ell-vs-ell-bar} and
	\ref{thm:eF-close-to-Hv} to get
	\[ (1-3s_m)(1-6s_m)e^F_{\Delta,n}\ell_{ij}^2 \leq \tilde{\ell}_{ij}^2 \leq
	(1+3s_m)(1+6s_m)e^F_{\Delta,n}\ell_{ij}^2, \] or, more compactly
	\begin{equation}
		\label{eq:g_ii}
		\left| e^F_{\Delta,n}\ell_{0i}^2- \tilde{\ell}_{0i}^2 \right| \leq
		18s_me^F_{\Delta,n}\ell_{0i}^2.
	\end{equation}
	
	For the off-diagonal terms, we have the following calculation:
	\begin{align*}
		\left| e^F_{\Delta,n}g_{ij}-\tilde{g}_{ij} \right|
		&= \left| \frac{1}{2} \left( e^F_{\Delta,n}\ell_{0i}^2-\tilde{\ell}_{0i}^2
		+e^F_{\Delta,n}\ell_{0j}^2-\tilde{\ell}_{0j}^2 -e^F_{\Delta,n}\ell_{ij}^2+\tilde{\ell}_{ij}^2 \right) \right| \\
		&\leq \frac{1}{2} \left( \left| e^F_{\Delta,n}\ell_{0i}^2-\tilde{\ell}_{0i}^2 \right|
		+ \left| e^F_{\Delta,n}\ell_{0j}^2-\tilde{\ell}_{0j}^2 \right|
		+ \left| e^F_{\Delta,n}\ell_{ij}^2-\tilde{\ell}_{ij}^2 \right|\right) \\
		&\leq \frac{1}{2} \left( 18s_me^F_{\Delta,n}\ell_{0i}^2+18s_me^F_{\Delta,n}\ell_{0j}^2
		+18s_me^F_{\Delta,n}\ell_{ij}^2 \right) \\
		& \leq 27s_me^F_{\Delta,n}\epsilon_n^2\\
		&\leq \frac{\mu \lambda_{\min}}{2},
	\end{align*}
%
where
	$\mu:= \frac{216s_m}{\vartheta^2}$ and $\lambda_{\min}$ is the smallest eigenvalue of $g^\Delta$, and we have used (\ref{eq:eigenvaluebound}).
%
	We can apply Lemma
	\ref{thm:barycentric-lemma6} to get that
	\begin{align*}
		\left| |X|^2_{e^F_{\Delta,n}g^{\Delta}}- |X|^2_{\phi_n^{*}{\tilde{g}^{\Delta}}} \right|
		&\leq \mu |X|^2_{e^F_{\Delta,n}g^{\Delta}} \\
		& = \frac{216s_m}{\vartheta^2} |X|^2_{e^F_{\Delta,n}g^{\Delta}}.
	\end{align*}
	Letting $C:= \frac{216}{\vartheta^2}$ and rewriting as a two-sided bound gives
	the result.
\end{proof}

\begin{remark}
	Notice that the metrics $g^{\Delta}_n$ and $\phi_n^{*}{\tilde{g}^{\Delta}}_n$ are piecewise flat but the metrics $e^F_{\Delta,n} g^{\Delta}_n$ are only \emph{conformal to} piecewise flat metrics.
\end{remark}

\section{Convergence}
\label{chap:main-result}

We begin by introducing the class of sequences.

\begin{definition}
	\label{def:admissible-sequence}
	Let $(M,g)$ and $(\tilde{M}, \tilde{g})$ be complete Riemannian surfaces and let
	$\Omega\subset M$ and $\tilde{\Omega}\subset \tilde{M}$ be embedded submanifolds such
	that $\Omega$ and $\tilde{\Omega}$ are diffeomorphic. Let $\{(\Omega_n,
	T_n,\epsilon_n)\}$ and $\{(\tilde{\Omega}_n,T_n,\epsilon_n)\}$ be generalized triangulated
	exhaustions of $\Omega$ and $\tilde{\Omega}$ respectively such that the triangulations $T_n$ are combinatorially identical. (See Definition \ref{def:exhaustion}.)
	
	Let $\{f_n\}$ and $\{\tilde{f}_n\}$ be sequences of discrete conformal factors
	for $\Omega^\Delta_n$ and $\tilde{\Omega}^\Delta_n$ respectively such that
	for each $n$ the piecewise flat manifolds $(\Omega_n,T_n,\ell(f_n))$ and
	$(\tilde{\Omega}_n,T_n,\ell(\tilde{f}_n))$ are discrete conformal under the
	discrete conformal structure $\mathcal{C}$ with discrete conformal map $\phi_n:\Omega_n \to \tilde{\Omega}_n$.
	
	Assume that the following conditions hold:
	\begin{enumerate}
		\item The sequence $\{T_n,f_n, \tilde{f}_n\}$ satisfies
		LDCR, Condition \ref{def:LDCR}.
		\item For each compact $K\subset \Omega$, the ratio of discrete conformal
		factors $e^{\tilde{f}_n(v)}/e^{f_n(v)}$ has a uniform upper bound $H=H(K)$
		depending only on the compact set $K$.
		\item There is some point $x\in\Omega$ such that the image set
		$\{\Phi_n(x)\}$ is contained in some compact subset $V\subset N$.
	\end{enumerate}
	In this case, we say that $\{(\Omega_n,\tilde{\Omega}_n, T_n,
	f_n,\tilde{f}_n, \epsilon_n)\}$ is an \emph{admissible sequence for $(\Omega,\tilde{\Omega})$}.
	
	If in addition the LDCR constants $s_m$ are such that there is some positive
	constant $\alpha$ such that $s_m\leq \alpha/m$, then we say that
	$\{(\Omega_n,\tilde{\Omega}_n,T_n,f_n,\tilde{f}_n, \epsilon_n)\}$ is a \emph{proper}
	admissible sequence for $(\Omega,\tilde{\Omega})$.
\end{definition}

\begin{remark}
	The assumption for proper is based on an estimate for hexagonal and bounded valence circle packings that is well-studied, e.g., \cite{rodin-sullivan, He-hex-packing,he-rodin-bdd-valence,doyle1994asymptotic}. It is also natural in the setting that the mappings come from a smooth map, as described in Section \ref{sec:discussion} in the discussion of B\"ucking's work. It is possible a weaker assumption may suffice, e.g., \cite{rodinschwarz2}, or an alternate argument that uses hyperbolic or other rigidity as suggested by \cite{schramm1991rigidity,he-schramm96,stephenson1996probabilistic}. One difficulty is that many of these methods use interstitial maps that are most natural on circle packings and don't obviously generalize to other discrete conformal structures. Properness will not be used until Section \ref{sec:eF}.
\end{remark}

We will now prove Theorem \ref{thm:main-thm} in two parts as Theorems \ref{thm:convergenceofmaps} and \ref{thm:convergenceofpullback}.

\subsection{Convergence of maps}
We begin with the main estimate.

\begin{proposition}
	\label{thm:main-prop}
	Assume $(M,g)$ and $(\tilde{M},\tilde{g})$ are Riemannian surfaces and let $\Omega\subset M$ and
	$\tilde{\Omega}\subset \tilde{M}$ be diffeomorphic embedded submanifolds with
	admissible sequence $\{(\Omega_n,\tilde{\Omega}_n,T_n,f_n,\tilde{f}_n, \epsilon_n)\}$. Let $K \subset \Omega$ be compact and take $p\in K$.
	
	Fix $N>0$ such that for all $n>N$, $p \in \Omega_n$ each such $n$, the closest vertex $v_n$ to $p$ is
	the center of a realized closed combinatorial disk of radius $m$ in $T_n$.
	
	Then there is some constant $C=C(K)$ for which the following estimate holds:
	\begin{equation}
		\label{eq:main-estimate}
		(1-\beta\epsilon_n)^2(1-Cs_m)e^F_n|X|^2_g\leq |X|^2_{\Phi_n^{*}\tilde{g}}
		\leq (1+\beta\epsilon_n)^2(1+Cs_m)e^F_n|X|^2_g,
	\end{equation}
	where $\beta=\beta(C_0,\vartheta)$ is the Riemannian barycentric constant from Corollary
	\ref{thm:bary-estimates} (for bounds $|\Rm_g|_g, |\Rm_{\tilde{g}}|_{\tilde{g}} \leq C_0$ and $\vartheta = \vartheta(K)$ from Definition \ref{def:admissible-sequence}), $\{\epsilon_n\}$ is the sequence of maximum edge
	lengths for the generalized triangulated exhaustions $\{(\Omega_n,T_n,\epsilon_n)\}$ of
	$\Omega$ and $\{(\tilde{\Omega}_n,T_n,\epsilon_n)\}$ of $\tilde{\Omega}$, and $\{s_m\}$
	is the sequence of LDCR constants.
\end{proposition}

\begin{proof}
%
	First we rewrite $|X|^2_{\Phi_n^{*}\tilde{g}}$ as
	\begin{align*}
		|X|^2_{\Phi_n^{*}\tilde{g}}
		&= \left| \left( \tilde{\Psi}_n\circ \phi_n\circ \Psi_n^{-1} \right)_{*}X \right|^2_{\tilde{g}} \\
		&= \left| \left( \phi_n\circ \Psi_n^{-1} \right)_{*}X \right|^2_{\tilde{\Psi}_n^{*}\tilde{g}}
	\end{align*}
	and then use \eqref{eq:Psi-estimate} from Corollary \ref{thm:bary-estimates} to
	get that
	\begin{equation}
		\label{eq:main-estimate-Psi}
		(1-\beta\epsilon) \left| \left( \phi_n\circ \Psi_n^{-1} \right)_{*}X \right|^2_{\tilde{g}^{\Delta}}
		\leq |X|^2_{\Phi_n^{*}\tilde{g}}\leq
		(1+\beta\epsilon) \left| \left( \phi_n\circ \Psi_n^{-1} \right)_{*}X \right|^2_{\tilde{g}^{\Delta}}.
	\end{equation}
	
	Next we rewrite $\left| \left( \phi_n\circ\Psi_n^{-1} \right)_{*}X
	\right|^2_{\tilde{g}^{\Delta}}$ as
	\[ \left| \left( \phi_n\circ\Psi_n^{-1} \right)_{*}X \right|^2_{\tilde{g}^{\Delta}} =
	\left| \left( \Psi_n^{-1} \right)_{*}X \right|^2_{\phi_n^{*}\tilde{g}^{\Delta}}\]
	and use Theorem \ref{thm:phi-n-metric-estimate} to estimate the right hand
	side. When combined with \eqref{eq:main-estimate-Psi}, this shows that
	\begin{align*}
		\label{eq:main-estimate-phi}
		(1-\beta\epsilon)(1-Cs_m) \left| \left( \Psi_n^{-1} \right)_{*}X \right|^2_{e^F_{\Delta}g^{\Delta}}
		\leq \left| X \right|^2_{\Phi_n^{*}\tilde{g}} \leq
		(1+\beta\epsilon)(1+Cs_m) \left| \left( \Psi_n^{-1} \right)_{*}X \right|^2_{e^F_{\Delta}g^{\Delta}}.
	\end{align*}
	
We can now use Corollary \ref{thm:bary-estimates} again to get
	\[ (1-\beta\epsilon)^2(1-Cs_m)e^F_n|X|^2_g \leq |X|^2_{\Phi_n^{*}\tilde{g}} \leq
	(1+\beta\epsilon)^2(1+Cs_m)e^F_n|X|^2_g.\]
\end{proof}

\begin{theorem}	\label{thm:convergenceofmaps}
	Let $\{(\Omega_n,\tilde{\Omega}_n,T_n,f_n,\tilde{f}_n, \epsilon_n)\}$ be an admissible
	sequence for $(\Omega, \tilde{\Omega})$ and let $\{\Phi_n\}$ be the
	corresponding sequence of barycentric discrete conformal maps. Then the family
	$\{\Phi_n\}$ has a subsequence that converges uniformly on compact subsets of
	$\Omega$. 
\end{theorem}

\begin{proof}
	We begin by showing that the family $\{\Phi_n\}$ is equicontinuous on compact
	subsets of $\Omega$. Let $K\subset \Omega$ compact and let $L\in\N$ be large
	enough so that $K\subset \Omega_n$ for all $n\geq L$. Take $p,q\in K$. Then by
	Proposition \ref{thm:main-prop}, we have that
	\[ |X|_{\Phi_n^{*}h} \leq
	(1+\beta\epsilon_n)\sqrt{1+Cs_m}\sqrt{e^F_n}|X|_g. \] Since $\{\epsilon_n\}$
	is a decreasing sequence we can easily find $\epsilon>0$ such that
	$\epsilon_n < \epsilon$ for every $n\geq L$. Further, we can assume that there
	is some $m$ such that every vertex in $K$ is the center of a combinatorial
	closed disk of at least $m$ generations in $\Omega_n$ for every $n$. Hence
	$s_m$ does not depend on $n$. Finally, since we are assuming $H_n(v)\leq H_K$
	for every vertex $v$ in $K$, we have that $e^F_n(p)\leq H_K^2$ for every point
	$p\in K$. Hence
	\begin{equation}
		\label{eq:main-thm-middle}
		|X|_{\Phi_n^{*}h} \leq (1+\beta\epsilon)\sqrt{1+Cs_m}H_K|X|_g.
	\end{equation}
	
	Now each $\Phi_n$ is a homeomorphism on $K$ and Lemma
	\ref{thm:Phi-n-diffeo-closed-simplices} shows that $\Phi_n$ is a
	diffeomorphism when restricted to each (closed) simplex. Hence we can use
	Proposition \ref{thm:integral-estimate} from the Appendix along with
	\eqref{eq:main-thm-middle} to see that
	\[ d_h(\Phi_n(p),\Phi_n(q))\leq H_K(1+\beta\epsilon)\sqrt{1+Cs_m}d_g(p,q).\]
	
	This means that each $\Phi_n$ is Lipschitz with the same Lipschitz constant
	and so the family $\{\Phi_n\}$ is equicontinuous on $K$.
	
	Next we use the Arzel\`a-Ascoli Theorem, Proposition \ref{thm:arzela-ascoli-manifolds}
	in the Appendix, to show that $\{\Phi_n\}$ has a convergent subsequence. We
	have already shown that the family $\{\Phi_n\}$ is equicontinuous, so next we
	need to show that for each $p\in \Omega$ the image set $S(p) = \{\Phi_n(p)\, :\, p\in
	\Omega_n\}$ is contained in a compact set. To do so, let
	$x\in \Omega$ be the point such that $\{\Phi_n(x)\}$ is contained
	in some compact $V\subset N$. This point was assumed to exist in the
	definition of admissible sequence.
	
	Let $p\in \Omega$ and take $K$ to be a compact set in $\Omega$ containing $p$.
	Then for $n,m > L$, we can calculate as follows:
	\begin{align*}
		d_{\tilde{g}}(\Phi_n(p), \Phi_m(p))
		& \leq d_{\tilde{g}}(\Phi_n(p), \Phi_n(x)) + d_{\tilde{g}}(\Phi_n(x), \Phi_m(x)) + d_{\tilde{g}}(\Phi_m(x), \Phi_m(p)) \\
		& \leq H_K(1+\beta)\sqrt{1+C}d_g(p,x) + \mathrm{diam}(V)+ H_K(1+\beta)\sqrt{1+C}d_g(x,p) \\
		&= 2H_K(1+\beta)\sqrt{1+C}d_g(p,x) + \mathrm{diam}(V) \\
		& \leq A,
	\end{align*}
	where $A>0$ does not depend on $n$. Hence the set $S(p)$ is contained in the
	closed ball $\bar{B_A(y)}$, which is a closed and bounded subset of a complete
	manifold and hence compact.
	
	Thus we have that the family $\{\Phi_n\}$ is equicontinuous and pointwise
	bounded on each compact $K\subset \Omega$ and hence has a convergent
	subsequence on each such $K$. Take an exhaustion
	of $\Omega$ by compact subsets $K_1\subset K_2\subset\dots$ and use a
	diagonalization argument to build a subsequence $\{\Phi_{n_i}\}$ such that
	$\Phi_{n_i}$ converges uniformly on each $K_j$ and hence the subsequence
	$\{\Phi_{n_i}\}$ converges on any compact subset.

\end{proof}

\begin{remark}
	This theorem could be generalized some by using a generalization of the ring lemma, as in \cite{rodin-sullivan}, rather than LDCR, a generalization of the more precise hexagonal packing lemma. Since we are mainly concerned with convergence to conformal, Theorem \ref{thm:convergenceofmaps} is stated with the LDCR instead, which will be necessary for Theorem \ref{thm:convergenceofpullback}.
\end{remark}

\begin{remark}
	Many of the previous convergence results (e.g., \cite{rodin-sullivan,gu_convergence}) use quasiconformality of piecewise linear maps to prove
	convergence to a conformal map. The result follows essentially from the fact that the triangles in the domain and range
	become the same, often equilateral, due to the hexagonal rigidity or other forms of LDCR. While this is true for our
	discrete conformal maps, the situation in Theorem \ref{thm:convergenceofmaps} is slightly more complicated because
	the domain and range may not be flat and there may be nontrivial topology. One might be able to 
	make these arguments work with a closer look at the barycentric maps and by using Beltrami differentials,
	but these directions have not been pursued here.
\end{remark}

\subsection{Ratio of discrete conformal factors converges}
\label{sec:eF}

In this section we will prove that the ratio of exponentials of discrete conformal factors
converges uniformly on compact subsets of a surface $M$ to a continuous
function. As a first step in that direction, we prove equicontinuity of a
certain family of functions. We begin by defining a few maps and setting some
notation. Recall the Definition \ref{def:exhaustion}.
%
%

Let $M$ be a smooth Riemannian surface and let $\Omega\subset M$ be a
submanifold that admits a generalized triangulated exhaustion
$\{(\Omega_n,T_n,\epsilon_n)\}$. Suppose there are sequences of discrete conformal factors
$\{f_n\}, \{{\tilde{f}_n}\}$ such that $\{T_n, f_n, {\tilde{f}_n}\}$
satisfies LDCR (Condition \ref{def:LDCR}) and, further, that the sequence
$\{s_m\}$ is of order $1/m$. That is, assume there is a constant $\alpha$ such
that $s_m\leq \alpha/m$ for all $m$.

Let $H_n:V(T_n)\to\R$ be the ratio of exponentials of discrete conformal factors in image and domain,
$H_n(v):= e^{{\tilde{f}_n}(v)}/e^{f_n(v)}$.
Define $e^F_{\Delta,n}:\Omega^\Delta_n\to\R$ to be the linear interpolation of $H_n^2$.
That is, $e^F_{\Delta,n}(a) = \sum_i\lambda^iH_n^2(v_i)$, where
$(\lambda^0,\lambda^1,\lambda^2)$ are the barycentric coordinates of $a$.
Finally, let $e^F_n:\Omega_n\to\R$ be the composition $e^F_n:= e^F_{\Delta,n}\circ
\Psi_n^{-1}$.

We prove the following proposition.

\begin{proposition}
	\label{thm:eF-equicontinuous}
	Let $\Omega\subset M$ be an embedded submanifold that admits a generalized
	triangulated exhaustion $\{(\Omega_n,T_n)\}$ and define $e^F_n: \Omega_n\to\R$
	as above. Then under the condition that for each compact subset $K\subset
	\Omega$ there is a positive constant $H_K$ such that $1/H_K\leq H_n(v)\leq H_K$ for
	every vertex $v$ and every $n$, the family $\{e^F_n\}$ is uniformly Lipschitz on compact subsets of $\Omega$.
\end{proposition}

The condition in the above proposition that $H_n(v)\leq H_K$ for every $v$ and
$n$ is not a very strong condition, following directly from a discrete Schwarz,

Before we prove Proposition \ref{thm:eF-equicontinuous}, we need the following lemma.

\begin{lemma}
	\label{thm:deF}
	For each $n$, let $\Omega^\Delta_n = (T_n, f_n)$ and
	$\tilde{\Omega}^\Delta_n = (T_n,\tilde{f}_n)$ be triangulated
	$(\vartheta,\epsilon_n)$-full piecewise flat surfaces sharing the same
	combinatorics and discrete conformal structure such that the
	sequence $\{(T_n, f_n,{\tilde{f}_n})\}$ satisfies $LDCR$ with $s_m \leq \alpha/m$ for some constant $\alpha$. Let
	$e^F_{\sigma}:\sigma\to\R$ be the restriction of $e^F_{\Delta,n}$ to a single
	(closed) simplex $\sigma = [v_0,v_1,v_2]$.
	
	Assume further that each vertex $v_i$ of $\sigma$ is the center of a realized
	closed combinatorial disk of generation $m$ in $T_n$. Then the following
	inequality holds, where $g^{\Delta}$ is the Euclidean metric on $\sigma$ with the specified edge lengths:
	\begin{equation}
		\label{eq:deF}
		\left| de^F_{\sigma} \right|^2_{g^{\Delta}}
		\leq \left( \frac{6 \alpha H_n^2(v_0)}{m\vartheta\epsilon_n} \right)^2.
	\end{equation}
\end{lemma}

\begin{proof}
	We will find it useful to use the unit simplex $D$ as the coordinate rather than the standard simplex $\Delta$ (see Section \ref{sec:bary-eucl}). 
	
	We can rewrite $e^F_{\sigma}$
	to be a function from the unit simplex $D$ by replacing $\lambda^0$ by $1-\lambda^1-\lambda^2$. In these
	coordinates, we have that
	\[ e^F_{\sigma}(\lambda^1,\lambda^2) = H_n^2(v_0) + \sum_{i=1}^2
	\lambda^i(H_n^2(v_i)-H_n^2(v_0)).\]

	A direct calculation shows that the inverse metric $\left( g^{\Delta}
	\right)^{-1}$ in the coordinate of $D$ can be written as
	\[ \left( g^{\Delta} \right)^{-1} = \frac{1}{4A^2}
	\begin{pmatrix}
		\ell_{02}^2 & -\ell_{01}\ell_{02}\cos \phi_{12} \\[.75em]
		-\ell_{01}\ell_{02}\cos \phi_{12} & \ell_{01}^2
	\end{pmatrix},\]
	where $\phi_{12}$ is the angle between $p_1-p_0$ and $p_2-p_0$, and $A$ is the
	area of $\sigma$.

	We can now calculate as follows, keeping in mind that every edge has length
	$\ell_{ij}\leq \epsilon_n$ and using the fact that $-2\cos\phi_{12}\leq 2$:
	\begin{align*}
		\left| d e^F_{\sigma} \right|^2_{g^{\Delta}}
		& = \left( \frac{\partial e^F_{\sigma}}{\partial \lambda^1} \right)^2 \left( g^{\Delta} \right)^{11}
		+ \left( \frac{\partial e^F_{\sigma}}{\partial \lambda^2} \right)^2 \left( g^{\Delta} \right)^{22}
		+2 \left( \frac{\partial e^F_{\sigma}}{\partial \lambda^1} \right)
		\left( \frac{\partial e^F_{\sigma}}{\partial \lambda^2} \right) \left( g^{\Delta} \right)^{12} \\
		&\leq \frac{\epsilon_n^2}{4A^2} \left[
		(H_n^2(v_1)-H_n^2(v_0))^2+(H_n^2(v_2)-H_n^2(v_0))^2 \right.\\
		& \qquad\qquad\qquad
		+\left. 2 \left( H_n^2(v_1)-H_n^2(v_0) \right)\left( H_n^2(v_2)-H_n^2(v_0) \right) \right].
	\end{align*}
	
	By Lemma \ref{thm:H(v)-close-to-H(w)}, we have that for $i=1,2$,
	$\left|H_n^2(v_i)-H_n^2(v_0)\right|\leq 3s_mH_n^2(v_0)$ since each $v_i$ is the center of a
	(realized) closed combinatorial disk of generation $m$. Hence we have that
	\[ \left| d e^F_{\sigma} \right|^2_{g^{\Delta}} \leq \frac{\epsilon_n^2
		9s_m^2H_n^4(v_0)}{A^2}. \]
	
	The simplex $\sigma$ is $(\vartheta,\epsilon_n)$-full and a 2-simplex, so
	\[ 2A\geq \vartheta\epsilon_n^2.\]
	Hence,
	\[ \left| de^F_{\sigma} \right|^2_{g^{\Delta}} \leq \frac{36
		s_m^2H_n^4(v_0)}{\vartheta^2\epsilon_n^2}.\]
	
	Finally, since $s_m\leq \alpha/m$, 
	\[ \left| de^F_{\sigma} \right|^2_{g^{\Delta}}
	\leq \frac{(36)(H_n^4(v_0))(\alpha^2/m^2)}{\vartheta^2\epsilon_n^2}
	= \left( \frac{6 \alpha H_n^2(v_0)}{m\vartheta\epsilon_n} \right)^2, \] as required.
	
\end{proof}

Now we can prove Proposition \ref{thm:eF-equicontinuous}.

\begin{proof}[Proof of Proposition \ref{thm:eF-equicontinuous}]
	Let $K\subset \Omega$ be compact and let $p,q\in K$. Take $n$ large enough so
	that $K\subset \Omega_n$ for every $n$ and define $a_n:= \Psi_n^{-1}(p)$ and
	$b_n:=\Psi_n^{-1}(q)$ to be the points in $\Omega^\Delta_n$ corresponding to points
	$p,q$ in $\Omega$. Then by definition, \[e^F_n(p) = (e^F_{\Delta,n}\circ
	\Psi_n^{-1})(p) = e^F_{\Delta,n}(a_n)\] and similarly, $e^F_n(q) =
	e^F_{\Delta,n}(b_n)$. Hence $\left| e^F_n(p)-e^F_n(q) \right| = \left|
	e^F_{\Delta,n}(a_n)-e^F_{\Delta,n}(b_n) \right|$.
	
	Take a minimizing geodesic $\gamma:[0,1]\to M^{\Delta}$ connecting $a_n$ and
	$b_n$ and let $\{\sigma_i\}$ be an ordered list of simplices intersecting
	$\gamma$. Assume that unnecessary simplices have been removed from this list
	in the manner detailed at the beginning of the proof of Proposition
	\ref{thm:integral-estimate} in the Appendix and just as in that proof, let
	$\{t_i\}_{i=1}^{Q-1}$ be the list of $t$-values where $\gamma$ leaves one
	simplex and enters the next, with $t_0:= 0$ and $t_Q:= 1$. Hence
	$a_n=\gamma(t_0)$ and $b_n=\gamma(t_Q)$.
	
	Note that if a point $x$ is in the intersection of two simplices $\sigma$ and
	$\tau$, then $e^F_{\Delta,n}(x)$ can be calculated in either $\sigma$ or $\tau$
	and the result is the same. That is, if $x\in \sigma\cap\tau$, then
	$e^F_{\Delta,n}(x)=e^F_{\sigma}(x)=e^F_{\tau}(x)$, so we can use the Triangle
	Inequality to perform the following calculation:
	\begin{align}
		\left| e^F_n(p)-e^F_n(q) \right|
		&= \left| e^F_{\Delta,n}(a_n)-e^F_{\Delta,n}(b_n) \right| \nonumber\\
		&= \left| \sum_{i=0}^Q \left( e^F_{\sigma_i}(\gamma(t_i))
		-e^F_{\sigma_i}(\gamma(t_{i+1})) \right) \right| \nonumber\\
		& \leq \sum_{i=0}^Q \left| e^F_{\sigma_i}(\gamma(t_i))-e^F_{\sigma_i}(\gamma(t_{i+1})) \right|.
		\label{eq:eF-middle-step}
	\end{align}
	
	Each term of \eqref{eq:eF-middle-step} can be calculated on a single
	simplex, so we can use the Mean Value Theorem on each term separately and then
	sum at the end. Since $e^F_{\sigma}$ is linear on $\sigma$, the differential
	$de^F_{\sigma}$ is constant, so the Mean Value Theorem says that if $x$ and
	$y$ are points in $\sigma$, then
	\begin{equation}
		\label{eq:eF-MVT}
		\left| e^F_{\sigma}(x)-e^F_{\sigma}(y) \right|_{\R}
		\leq \left| de^F_{\sigma} \right|_{g^{\Delta}} d_{g^{\Delta}}(x,y).
	\end{equation}
	Lemma \ref{thm:deF} gives us an upper bound on
	$|de^F_{\sigma}|_{g^{\Delta}}$ and we now use this upper bound with
	\eqref{eq:eF-middle-step} to get
	\begin{align*}
		\left| e^F_n(p)-e^F_n(q) \right|
		& \leq \sum_{i=0}^Q \left| e^F_{\sigma_i}(\gamma(t_i))-e^F_{\sigma_i}(\gamma(t_{i+1})) \right| \\
		&\leq \sum_{i=0}^Q \left| de^F_{\sigma_i} \right|_{g^{\Delta}}d_{g^{\Delta}}(\gamma(t_i), \gamma(t_{i+1})) \\
		&\leq \alpha \sum_{i=0}^Q
		\frac{H_n^2(v_i)}{m_i\vartheta\epsilon_n}d_{g^{\Delta}}(\gamma(t_i),\gamma(t_{i+1})),
	\end{align*}
	where the vertex $v_i$ is one of the vertices of $\sigma_i$ and $m_i$ is the
	number of generations of the largest realized closed combinatorial disk
	about $v_i$ that is contained in $\Omega$.
	
	Next let $R=R(K)$ be the value guaranteed by Lemma \ref{thm:ball-radius-R}
	such that for each point $p\in K$ there is a geodesic ball of radius $R$ about
	$p$ completely contained in $\Omega$. Now by Lemma \ref{thm:R-m-epsilon}, we have
	that $m_i\geq R/(2\epsilon_n)$ and we are assuming that $H_n(v_i)\leq H_K$ for
	every vertex $v_i$ contained in $K$, so we have the following:
	\begin{align*}
		\left| e^F_n(p)-e^F_n(q) \right|
		&\leq 6\alpha \sum_{i=0}^Q
		\frac{H_n^2(v_i)}{m_i\vartheta\epsilon_n}d_{g^{\Delta}}(\gamma(t_i),\gamma(t_{i+1})) \\
		&\leq \frac{12\alpha H_K^2}{R\vartheta} \sum_{i=0}^Q d_{g^{\Delta}}(\gamma(t_i),\gamma(t_{i+1})) \\
		& =\frac{12\alpha H_K^2}{R\vartheta} d_{g^{\Delta}}(a_n,b_n),
	\end{align*}
	where the last equality follows since $\gamma$ is a minimizing geodesic from
	$a_n$ to $b_n$.
	
	Finally, Proposition \ref{thm:integral-estimate} says that
	$d_{g^{\Delta}}(a_n,b_n)\leq (1+\beta\epsilon)d_g(p,q)$, so
	\[ \left| e^F_n(p)-e^F_n(q) \right|\leq \frac{12H_K^2a}{\vartheta
		R}d_{g^{\Delta}}(a_n,b_n) \leq \frac{12H_K^2a}{\vartheta
		R}(1+\beta\epsilon)d_g(p,q),
	\] 
	and hence every $e^F_n$ is Lipschitz with
	the same constant $L\leq 12\alpha H_K^2(1+\beta)/(\vartheta R)$ and the
	family $\{e^F_n\}$ is uniformly equicontinuous.
\end{proof}

We now have a direct corollary of the previous result.

\begin{corollary}
	\label{thm:eFn-to-eF}
	There exists a subsequence $\{e^F_{n_k}\}\subset \{e^F_n\}$ and a positive continuous
	function $e^F:\Omega\to \R $ such that $e^F_{n_k}\to e^F$ uniformly on compact
	subsets of $\Omega$.
\end{corollary}

\begin{proof}
	This proof is an application of the Arzel\'a-Ascoli Theorem.	
	By Proposition \ref{thm:eF-equicontinuous}, the collection $\{e^F_n\}$ is
	equicontinuous on compact subsets, so all that remains is to show that
	$\{e^F_n\}$ is pointwise bounded. That is, that for any $p\in \Omega$,
	$|e^F_n(p)|\leq C$ for some constant $C$ that does not depend on $n$. But this
	follows immediately from the assumption that $H_n(v) =
	e^{{\tilde{f}_n}(v)}/e^{f_n(v)}\leq H_K$ for every vertex $v$ and every $n$.
	
	Hence if $K$ is an arbitrary compact subset of $\Omega$, then there is a
	subsequence of $\{e^F_n\}$ that converges uniformly on $K$.
	
	All that remains is a diagonalization argument to show that there is a
	subsequence $\{e^F_{n_k}\}$ that is convergent on \emph{every} compact subset
	of $\Omega$.
	
	Take a countable exhaustion of $\Omega$ by compact subsets. That is, let
	$\{K_i\}_{i=1}^\infty$ be compact subsets of $\Omega$ such that $K_i\subset
	K_{i+1}$ and for each point $p\in \Omega$, there is a number $Q\in \N$ such that
	$p\in K_i$ for every $i\geq Q$.
	
	Let $\{e^F_{1,n}\}\subset \{e^F_n\}$ be a subsequence that converges uniformly
	on $K_1$. Then take a further subsequence $\{e^F_{2,n}\}\subset \{e^F_{1,n}\}$
	which converges uniformly on $K_2$. Proceed in this manner indefinitely.
	
	Finally, build a new subsequence $\{e^F_i\}$ as follows. Let
	$e^F_1=e^F_{1,1}$, the first element in the first subsequence,
	$e^F_2=e^F_{2,2}$, the second element in the second, and so on. Then since
	$\{e^F_i\}\subset \{e^F_{1,n}\}$, the subsequence $\{e^F_i\}$ converges in
	$K_1$. Similarly $\{e^F_i\}_{i=2}^{\infty}\subset \{e^F_{2,n}\}$ and hence
	$\{e^F_i\}$ converges in $K_2$ as well. By the same argument, $\{e^F_i\}$
	converges uniformly on each of the countably many subsets $\{K_i\}$, and hence
	on any compact subset of $\Omega$.
\end{proof}

\begin{theorem}\label{thm:convergenceofpullback}
		Let $\{(\Omega_n,\tilde{\Omega}_n,T_n,f_n,\tilde{f}_n, \epsilon_n)\}$ be a proper admissible
	sequence for $(\Omega, \tilde{\Omega})$ and let $\{\Phi_n\}$ be the
	corresponding sequence of barycentric discrete conformal maps. Then there exists
	a positive continuous function $e^F$ such that $\Phi_n^{*}\tilde{g}\to e^Fg$ in $L^{\infty}$ on compact
	subsets of $\Omega$, and hence convergence is to a conformal map.
\end{theorem}

\begin{proof}
	We show that $\Phi_n^{*}{\tilde{g}}\to e^Fg$ in $L^{\infty}$ on compact subsets
	of $\Omega$.
	
	Let $K\subset \Omega$ be compact and take $L\in \N$ to be large enough so that
	$K\subset \Omega_n$ for every $n>L$. Note that by discarding elements of
	$\{\Omega_n\}$, we may assume without loss of generality that $\Omega_n\subset
	\Omega_{n+1}$ for every $n$. Take $R>0$ such that $B_R(p)$ is a geodesic ball
	of radius $R$ contained in $\Omega_L$ and hence $B_R(p)$ is also contained in
	$\Omega_n$ for $n >L$. Such an $R$ exists by Lemma \ref{thm:ball-radius-R}
	along with our earlier assumption that the $\Omega_n$'s are nested.
	
	Define the set $E\subset M$ to be $E:= \bigcup_nE(T_n)$, the union over all $n$
	of (closed) edges in the triangulations $T_n$. Note that for each $n$, the set of edges
	$E(T_n)$ is a set of measure zero (in $M$) and since there are countably many
	such subsets, the countable union $E$ is also a set of measure zero.
	
	Let $X$ be a vector field. Then by Proposition \ref{thm:main-prop}, the
	following inequality holds for any point in $\Omega\setminus E$, where each
	vertex $v$ in $T_n$ is the center of a closed combinatorial disk of $m$
	generations
	\begin{equation}
		\label{eq:Phi-n-almost-conformal-middle-step}
		\left( 1-\beta\epsilon_n \right)^2 \left(1 - Cs_m
		\right)e^F_n|X|^2_g \leq |X|^2_{\Phi_n^{*}{\tilde{g}}} \leq (1+\beta\epsilon_n)^2\left(1+
		Cs_m\right) e^F_n|X|^2_g.
	\end{equation}
	
	By assumption $\epsilon_n\to 0$ as $n\to\infty$, so both
	$(1-\beta\epsilon_n)^2$ and $(1+\beta\epsilon_n)^2$ approach $1$ as
	$n\to\infty$. Corollary \ref{thm:eFn-to-eF} says that $e^F_n\to e^F$ uniformly
	as $n\to\infty$. Note that this is the only place we essentially need the properness assumption.
	
	All that remains is to show that as $n\to\infty$, $1-Cs_m$
	and $1+Cs_m$ both approach $1$. 
	By Lemma \ref{thm:R-m-epsilon}, we
	have that \[ \frac{1}{m}\leq \frac{2\epsilon_n}{R}, \] which implies
	$m\to \infty$ as $n\to \infty$ and so $s_m \to 0$.

	
	We have shown that on $K\setminus E$, both the left hand and right hand sides
	of \eqref{eq:Phi-n-almost-conformal-middle-step} converge uniformly to
	$e^F|X|^2_g$, so clearly the middle term, $|X|^2_{\Phi_n^{*}{\tilde{g}}}$ must as well.
	Hence, 
	$\Phi_n^{*}{\tilde{g}}\to e^Fg$ in
	$L^{\infty}$.
\end{proof}

\section{Circle packing convergence of simply connected domains} \label{sec:circlepack}

In this section we use Theorem \ref{thm:main-thm} to prove the Rodin-Sullivan Theorem \cite{rodin-sullivan}. 
Recall the definitions of the circle packing conformal structure from Section \ref{sec:discreteconformal map}.



Let $(\C,g)$ be the complex plane with the usual Euclidean metric $g$ and let
$\Omega\subset \C$ be a simply connected bounded region in the plane. For each
$n>0$ let $T_{n}$ be the circle packing with hexagonal combinatorics and circles of radius $1/n$ intersected with $\Omega$. 
Let $\Omega_n =|T_{n}|$, where $|T|$ denotes the carrier of the triangulation $T$.
We define $\tilde{\Omega}_n$ by taking the combinatorics of $T_n$ and changing the radii so that boundary circles are internally tangent to the unit disk
and the internal circles pack the disk (see \cite{rodin-sullivan,circle_packing} for existence). The circle packing maps are the piecewise linear maps $\phi_n:\Omega_n \to \tilde{\Omega}_n$.

\begin{lemma}
	\label{thm:RS-gen-trian-exh}
	The sequence $\{(\Omega_{n},T_{n})\}$ is a generalized
	triangulated exhaustion of $\Omega$.
\end{lemma}

\begin{proof}
%
	Since the radii of circles in $\Omega_{n}$ goes to 0 as $n\to \infty$
	does, every compact subset $K\subset \Omega$ is eventually contained in
	$\Omega_{n}$ for all $\epsilon$ small enough. Further, each
	$\Omega_{n}$ lies within $\Omega$, and hence necessarily within any open
	set $U\supset \Omega$. This shows the first two conditions of Definition
	\ref{def:exhaustion}
	
	Finally we note that since every triangle in $T_{n}$ is equilateral
	with edge length $2/n$, each triangle is $(\sqrt{3}/2, 2/n)$-full
	and hence the third condition is satisfied as well and $\{(\Omega_{n},
	T_{n})\}$ is a generalized triangulated exhaustion of $\Omega$.
\end{proof}

%


Before we can prove convergence of the ranges in Lemma \ref{thm:RS-image-gen-trian-exh}, we first must prove
the following proposition.

\begin{proposition}
	\label{thm:image-fullness-prelim-prop}
	Given a compact set $K\subset \Omega$, let $\tilde{R}_n$ be the largest radius of
	any circle in the image packing that corresponds to a vertex in $T_n$ adjacent to a
	vertex in $K$. There exists a constants $C$ and $N$ depending only
	on $K$ such that for every vertex $v$ that is adjacent to a vertex in $K$, if $n\geq N$ the
	radius $\tilde{r}_{n}(v)$ in the image of the circle packing map satisfies:
	\[ \tilde{r}_{n}(v) \geq C\tilde{R}_n.\]
\end{proposition}

\begin{proof}
	Our argument is inspired by the proof of Proposition 1 in \cite{rodinschwarz2}.

	Let $d$ denote the distance from $K$ to the boundary of $\Omega$, $d =
d(K,\partial \Omega)$. For any $n>1/(2d)$ there is a $\gamma_n>0$ such that $d/2 < 2\gamma_n/n < d$, so for every vertex $v$ in $K$ there is a hexagon of $\gamma_n$ generations centered at $v$. Call this hexagon $H_{m,\gamma_n}(v)$. By the Hexagonal Packing Lemma, since it is the center of $H_{n,\gamma_n}(v)$, if $v$ and $w$ are neighboring vertices, we must have 
\[
\tilde{r}_n(v) \geq (1-s_{\gamma_n})\tilde{r}_n(w)\geq (1-s_{nd/4})\tilde{r}_n(w).
\] 

If $D$ is the diameter of $K$, we also have that for any two vertices $v$ and $w$ in $K$, there is a path connecting centers of neighboring circles of length less than $2D$ with fewer than $2Dn$ circles. Let $w$ be the vertex with the maximal radius, i.e., $\tilde{r}_n(w)=\tilde{R}_n$. Since by Theorem \ref{thm:hexagonalpackinglemma} there exists $B$ such that $s_m \leq B/m$, it then follows that
\begin{align*}
	\tilde{r}_n(v) &\geq (1-s_{nd/4})^{2Dn}\tilde{R}_n \\
	&\geq \left(1-\frac{B}{nd/4}\right)^{2Dn}\tilde{R}_n\\
	&\geq \frac{1}{2} e^{-8BD/d}\tilde{R}_n,
\end{align*}
if $N$ is large enough. Letting $C:= \frac{1}{2}e^{-8BD/d}$ gives the result.
\end{proof}

\begin{lemma}
	\label{thm:RS-image-gen-trian-exh}
	There exists a sequence $\epsilon_n \to 0$ such that the sequence $\{(\tilde{\Omega}_n, T_{n}, \epsilon_n)\}$ is a generalized
	triangulated exhaustion of $\D$.
\end{lemma}

\begin{proof}
	First, note that the second condition in Definition
	\ref{def:exhaustion} is trivially satisfied since the carrier
	$\tilde{\Omega}_n$ is contained in $\D$ for every $n$.
	
	To show the first condition holds we use the Length-Area Lemma, Lemma
	\ref{thm:length-area} of the Appendix, to show
	that the radii of boundary circles converge uniformly to 0, at which point it
	is clear that any compact $K\subset \D$ is contained in $\tilde{\Omega}_n$ for
	all $n$ sufficiently large. This argument uses the structure of the hexagonal lattice and the divergence of the harmonic series (see \cite{rodin-sullivan} or \cite[p. 252]{circle_packing} for details). It follows that the maximum radius $\tilde{R}_n$ goes to zero as $n\to \infty$.
	
	Next fix a compact subset $V\subset \D$ and take $N$ large enough so
	that $V\subset \tilde{\Omega}_n$ for every $n>N$. We want to show that
	every triangle in $\tilde{\Omega}_n$ that has at least one vertex in $V$
	satisfies the fullness condition.
	
	Note that since the inverse circle packing maps $\phi_{n}^{-1}$ are
	continuous, the pre-image $\phi_{n}^{-1}(V)$ is compact. Define $K:=
	\phi_{n}^{-1}(V)$. By Proposition \ref{thm:image-fullness-prelim-prop}, we
	know that if $\sigma = [v_0,v_1,v_2]\subset \tilde{\Omega}_n$ is a triangle with one of its vertices
	contained in $V$ then the radii $\tilde{r}_{n}(v_i)$ corresponding to the
	vertices of $\sigma$ satisfy $\tilde{r}_{n}(v_i) \geq C \tilde{R}_n$.
	
	At this point we can use Heron's formula to estimate the area $A$ of $\sigma$ as follows:
	\begin{align*}
		A &= \sqrt{\tilde{r}_{n}(v_0)\tilde{r}_{n}(v_1)\tilde{r}_{n}(v_2)
			\left( \tilde{r}_{n}(v_0)+\tilde{r}_{n}(v_1) + \tilde{r}_{n}(v_2)\right)} \\
		& \geq \sqrt{3}C^2\tilde{R}_{n}^2.
	\end{align*}
	Since $\tilde{R}_{n}$ is the largest radius of any vertex adjacent to a vertex
	in $V$, we have shown that each triangle $\sigma$ that has at least one vertex
	in $V$ is $(\vartheta, \epsilon_n)$-full (recall Definition
	\ref{def:fullness}) where $\vartheta =\sqrt{3}C^2/2$ and $\epsilon_n = 2\tilde{R}_{n}$, which goes to zero as $n\to \infty$. Hence the sequence
	$\{(\tilde{\Omega}_n, T_{n}, \epsilon_n)\}$ is a generalized triangulated exhaustion
	of $\D$.
\end{proof}

Now that we have shown that there is a generalized triangulated exhaustion of
$\D$, we next show that there is a proper admissible sequence for $(\Omega,\D)$.

\begin{lemma}
	\label{thm:RS-admiss-seq}
	There is a proper admissible sequence $\{(\Omega_n, \tilde{\Omega}_n,
	T_{n}, f_n, \tilde{f}_n)\}$ for $(\Omega, \D)$, where $f_n = \log \frac{1}{n}$ and $\tilde{f}_n = \log \tilde{r}_n$ gotten by the Circle Packing Theorem, Theorem \ref{thm:hyp-circle-unif}.
\end{lemma}

\begin{proof}
	Firstly, $\Omega$ and $\D$ are diffeomorphic embedded submanifolds of the
	complex plane $\C$, which is a Riemannian surface. We have already shown that
	$\{(\Omega_n, T_{n})\}$ and $\{(\tilde{\Omega}_n, T_{n})\}$
	are generalized triangulated exhaustions of $\Omega$ and $\D$ respectively,
	with the same underlying (combinatorial) triangulation $T_{n}$.
	
	Note that $\{(T_{n},f_n, \tilde{f}_n)\}$
	satisfies the LDCR condition by the Hexagonal Packing Lemma, stated in more generality as Theorem \ref{thm:hexagonalpackinglemma} in the Appendix.
%
	The second condition in Definition \ref{def:admissible-sequence} holds in this
	context by Corollary \ref{thm:circles-HK-bound} in the Appendix, following from a discrete analogue of the Schwarz Lemma, Theorem \ref{thm:schwarz-lemma-circles}.
	
	The third condition follows from the normalization assumption on
	the circle packing maps $\phi_n$. Explicitly, the distinguished point
	$z_0$ is mapped under $\phi_n$ to a point within the flower of the image
	circle $c_0'$, which is centered at the origin. Using the Hexagonal Packing
	Lemma,
	we could find a (closed) ball
	$\bar{B}_{R_{\epsilon}}(0)$ about the origin whose radius $R_{\epsilon}$
	decreases with $\epsilon$ such that $\phi_n(z_0)$ lies within this
	closed ball. However, it is simpler to note that $\phi_n(z_0)$ will
	always lie within the unit disk $\D$ and the closure $\bar{\D}$ is compact in
	$\C$. Hence the image set $\{\phi_n(z_0)\}$ is contained in a compact
	subset of $\C$, as required.
	
	At this point we have shown that the sequence $\{(\Omega_n, \tilde{\Omega}_n,
	T_{n}, f_n, \tilde{f}_n)\}$ is an admissible sequence for
	$(\Omega, \D)$. In addition, we want this admissible sequence to be
	\emph{proper}. That is, we want the LDCR constants $s_m$ to be such that
	$s_m\leq A/m$ for some positive constant $A$ independent of $m$. This is part of Theorem \ref{thm:hexagonalpackinglemma} in the Appendix.
\end{proof}

Since we have shown that circle packing maps on $\Omega$ induce a proper
admissible sequence for $(\Omega,\D)$, the convergence of circle packings maps follows from Theorem \ref{thm:main-thm}. For brevity we do not provide conditions for convergence to a unique map, but prefer to only show subsequential convergence to a Riemann mapping.

\begin{theorem}
	Let $\Omega\subset \C$ be a simply connected bounded region in the plane and
	let $\{\phi_n\}$ be the sequence of circle packing maps into the unit
	disk $\D$ taken from subsets of the hexagonal packing in $\Omega$ with radius $1/n$. Then $\{\phi_n\}$ has a subsequence that converges uniformly
	on compact subsets of $\Omega$ to a Riemann mapping $\phi:\Omega\to \D$.
\end{theorem}

\section{Discussion}\label{sec:discussion}

In this section we will discuss the relationship of Theorem \ref{thm:main-thm} to other convergence of discrete conformal mappings in the literature. After the convergence theorems of Rodin-Sullivan described above and its direct continuations (see \cite{he-schramm96,circle_packing}), the most significant work has been in regard to vertex scaling conformal structures. 

The challenge in using Theorem \ref{thm:main-thm} is finding the sequences, which must satisfy a number of conditions, including that the domain and range are exhausted, fullness is bounded below, the LDCR condition is satisfied, and the ratio of the exponentials of the conformal factors (sometimes called ratio of radii) is uniformly bounded. 
In the literature, there are two primary ways of finding appropriate sequences: by approximating from a known mapping or by triangulating the domain and finding the range triangulation through a uniformization procedure. We will describe some of the results in these two directions and relate them to the requirements of Theorem \ref{thm:main-thm}.

First, we consider the convergence result of B\"ucking in \cite{bucking15}. We suppose there exists a domain $D\subset\C$ and let $\Omega_n$ be subsets of $\frac{1}{n}L$ where $L$ is a fixed triangular lattice of the entire plane $\C$. We let $F:\Omega \to \C$ denote a conformal mapping. The subsets $\tilde{\Omega}_n$ are found by taking the vertex scalings of $\Omega_n$ so that $f_n(v) = |\log F'(v)|$ for all $v\in \partial \Omega_n$ and so that the curvatures are all zero in the interior. The existence of the scale factors $f_n$ is the first part of \cite[Theorem 1.2]{bucking15}.

%

By \cite[Theorem 1.2i]{bucking15}, we know that 
\begin{equation} \label{eq:fnnearlogF}
	\left|f_n(v) - \log |F'(v)|\strut \right| \leq \frac{C}{n^2}.
\end{equation} It follows that $\tilde{\Omega}_n$ exhaust $K$.

The existence of the conformal mapping implies LDCR in the following way. Given a compact set $K$ there exists $R$ such that $B(x,R)$ is contained in $\Omega$ for all $x\in K$. Thus if there are $m$ generations of the lattice $\frac{1}{n}L$ contained in $B(x,R)$, we must have $m/n<R$. 
Let $v$ and $w$ be adjacent vertices in $\Omega_n$.
By Taylor's Theorem we have that there exists a constant $M$ such that
\begin{equation} \label{eqn:taylorestlogF}
	\left|\log |F'(v)| - \log |F'(w)|\strut \right|\leq M|v-w| \leq \frac{M}{n}.
\end{equation}
We then know that
\begin{align*}
	|f_n(v) - f_n(w)| &\leq \left|f_n(v) - \log |F'(w)|\strut \right| + \left|\log |F'(v)| - \log |F'(w)|\strut \right| +\left|\log |F'(w)| - f_n(w)\strut\right| \\
	&\leq \frac{C}{n^2} + \frac{M}{n} +\frac{C}{n^2}\\
	&\leq \frac{C'}{n}
\end{align*}
for an appropriate constant $C'$.
It now follows that for $n$ large enough,
\[
	\left| \frac{e^{f_n(v)}}{e^{f_n(v)}} - 1 \right| \leq \frac{2C'}{n} \leq \frac{2C'R}{m} 
\]
so we can take $s_m = \frac{2C'R}{m}$.

Finally, the fullness condition follows from the estimate of angles given in \cite[p. 146]{bucking15} and (\ref{eqn:taylorestlogF}). Finally, we can use (\ref{eq:fnnearlogF}) to see that the edge lengths go to zero in $\tilde{\Omega}_n$. It follows that this is a proper admissible sequence, and convergence in consistent with the rest of \cite[Theorem 1.2]{bucking15}.

\begin{remark}
	In \cite{bucking08}, another convergence of discrete conformal mappings is studies that has a slower rate of convergence. This result uses a definition of discrete conformality different than the one we consider, allowing dual structures based on two sets of circle patterns. While this result does not fit into the current context, but it would be interesting to build structure in this more general setting (see also \cite{schrammquad}).
\end{remark}

Instead of prescribing the boundary data, one can also use a uniformization theorem like the Riemann Mapping Theorem. The following theorem proves convergence of extended Riemann mappings to a Jordan domain with three marked points. The extended Riemann mapping is defined to be the unique map from equilateral triangle $ABC$ with side length $1$ to the closure of the domain. 
\begin{theorem}[Theorem 6.4 of \cite{luo-sun-wu_convergence}]
	Let $\Omega \subset \C$ be a Jordan domain with $\{p,q,r\}\in \partial \Omega$. There exists a sequence of triangulated piecewise flat disks $\Omega_n$ which exhaust $\Omega$ such that the piecewise linear discrete uniformization maps $\phi_n: \Omega_n \to \triangle ABC$ converge uniformly to the extended Riemann mapping $(\Omega, (p,q,r)) \to \triangle ABC$.
\end{theorem}
In particular, this work relies on a rigidity result that implies LDCR, given in Lemma \ref{thm:VS-RS-LDCR} in the Appendix. Careful estimates on the angles in the chosen triangulation and conformal factors ensure fullness. Note that, as with most of the results for vertex scaling, the analogue of properness is not used since there is not currently a result estimating $s_n$ in Lemma \ref{thm:VS-RS-LDCR}. Convergence to a conformal map follows from an argument on quasiconformality. Analogous work in the context of inversive distance packing is given in \cite{chen2025convergenceinversivedistancecircle}.

%
%

We will now discuss the application to the context of Gu-Luo-Wu \cite{gu_convergence} and Luo-Wu-Zhu \cite{luo-wu-zhu_convergence,wu-zhu_convergence}, whose work is closely related to the type of convergence result began by Rodin-Sullivan, but in the vertex scaling context. 

In \cite{gu_convergence}, the main theorems can be stated as follows. The specifics of the definitions can be found in the reference, and are discussed after the Theorem statements without giving complete definitions.

\begin{theorem}[Theorem 5.1 of \cite{gu_convergence}]
	Given a Riemannian triangle $(S,g)$ of angles $\pi/3$ a the three vertices and a $(\delta,c)$-regular sequence of geodesic triangulations $T_n$ of $(S,g)$. Let $\Omega_n$ be the piecewise flat manifolds gotten by gluing Euclidean triangles with edge lengths $l_n$ determined by the geodesic edge lengths from $T_n$ in $(S,g)$. For a given $w_n\in V(T_n)^*$, let $\tilde{\Omega}_n$ denote the piecewise flat manifold gotten by gluing Euclidean triangles with edge lengths equal to $e^{(w_n(v)+w_n(v'))/2}l_n(vv')$ for each edge $vv'$. 
	
	Then there exists $w_n\in V(T_n)^*$ such that for sufficiently large $n$, $\tilde{\Omega}_n$ is isometric to the Euclidean equilateral triangle, $\tilde{\Omega}_n$ is is a $\frac{\delta}{2}$-triangulation, and the discrete conformal maps $\phi_n:\Omega_n \to \tilde{\Omega}_n$ converge to the uniformization map $\phi$ in the sense that 
	\[
		\lim_{n\to\infty} \max_{v\in V(T_n)} \left| \phi_n(v) - \phi(v) \strut \right| =0.
	\]
\end{theorem}

\begin{theorem}[Theorem 6.1 of \cite{gu_convergence}]
	Given a torus with a Riemannian metric $g$ and a uniformization metric $g^*$, let $T_n$ be a sequence of $(\delta,c)$-regular geodesic subdivisions and let $l_n$ denote the geodesic lengths of the edges with $g$. Then there exist sequences of $w_n\in V(T_n)^*$ such that for sufficiently large $n$, such that $(T_n,w_n*l)$ are $\delta/2$-triangulations of a flat torus whose metric converge to $g^*$.
\end{theorem}
The condition of being a $\delta$-triangulation assures that inner angles of triangles are uniformly acute. The condition of $(\delta, c)$-regular assures a $\delta$-triangulation as well as a uniform bound on the ratio of the lengths to a quantity going to zero as the sequence goes to infinity. These guarantee fullness and bounds on the ratio of exponentials of conformal factors. Once again, LDCR follows from Lemma \ref{thm:VS-RS-LDCR}.

%
The following theorem proves convergence of the conformal factors for tori.
\begin{theorem}[Theorem 1.5 of \cite{wu-zhu_convergence}]
	Suppose $(M,g)$ is a closed orientable smooth Riemannian surface of genus 1, and $\bar{u}=\bar{u}_{M,g}$ is a smooth function on $M$ such that $e^{2\bar{u}}g$ is flat and $\Area (M,e^{2\bar{u}}g)=1$. Assume $T$ is a geodesic triangulation of $(M,g)$, and $l\in \R^{E(T)}_{>0}$ denotes the edge length in $(M,g)$. Then for any $\epsilon>0$, there exists constants $\delta=\delta(M,g,\epsilon)$ and $C=C(M,g,\epsilon)$ such that if $(T,l)_E$ is $\epsilon$-regular and $|l|<\delta$ then
	\begin{enumerate}
		\item there exists a unique discrete conformal factor $u\in \R^{V(T)}$, such that $(T,u*l)_E$ is globally flat and $\Area ((T,u*l)_E)=1$, and
		\item $|u-\bar{u}|_{V(T)}| \leq C|l|$. 
	\end{enumerate}
\end{theorem}
Note that Theorem 1.3 of \cite{luo-wu-zhu_convergence} and Theorem 1.4 of \cite{wu-zhu_convergence} prove similar results for spherical and hyperbolic surfaces. It is noted that the metric structure of edges for discrete conformal structures with hyperbolic background are simply the same as those with Euclidean background with the Euclidean lengths $l_{ij}$ replaced with $\sin$ of the spherical lengths or $\sinh$ of the hyperbolic lengths, respectively. This approach is specifically emphasized in the reformulation of Theorem 1.7 in \cite{luo-wu-zhu_convergence}.
The condition of $\epsilon$-regular ensures that angles do not degenerate and the triangulation remains uniformly Delaunay, which can be used to assure fullness.

\section{Appendix}
In this section we collect several results used in the text. 
\subsection{Analysis}
The following is a version of the Arzel\`a-Ascoli Theorem used in Section \ref{chap:main-result}.
\begin{proposition}
	\label{thm:arzela-ascoli-manifolds}
	Let $M$ and $N$ be smooth manifolds and let $\mathcal{F}$ be an equicontinuous
	collection of mappings $f:M\to N$ such that for any point $p\in M$, the set
	$\{f(p)\, : f\in \mathcal{F}\}$ is contained in a compact subset of $N$. Then
	every sequence $\{f_n\}\subset \mathcal{F}$ has a subsequence that converges
	uniformly on every compact subset of $M$.
\end{proposition}

\subsection{Riemannian barycentric coordinates results}
The following are results on Riemannian barycentric coordinates.

\begin{lemma}[Lemma 3 of \cite{barycentric}]
	\label{thm:barycentric-lemma3}
	Let $p_0,\dots, p_n\in \R^m$ be the vertices of a $(\vartheta, \epsilon)$-full
	Euclidean $n$-simplex, and let $g_{ij}^{\Delta}=\left\langle p_i-p_0, p_j-p_0
	\right\rangle_{\R^m}$ denote the pullback of its metric to the unit simplex
	$D$. Then the eigenvalues $\lambda_k$ of $g^{\Delta}$ satisfy
	\[ \vartheta \epsilon n^{1-n}\leq \sqrt{\lambda_k}\leq \epsilon n. \]
\end{lemma}

\begin{lemma}[Lemma 6 of \cite{barycentric}]
	\label{thm:barycentric-lemma6}
	Let $g$ and $\bar{g}$ be inner products on $\R^n$ such that all eigenvalues of
	$g$ (with respect to the Euclidean inner product) are larger than
	$\lambda_{\min}>0$ and $|g_{ij}-\bar{g}_{ij}|\leq \mu
	n^{-1}\lambda_{\min}$. Then $|(g-\bar{g})(v,v)|\leq \mu |v|^2_g$.
\end{lemma}

\begin{proposition}
	\label{thm:integral-estimate}
	Let $(M,g)$ and $(N,h)$ be homeomorphic manifolds and let $F:M\to N$ be a
	homeomorphism. Suppose that $(M,g)$ admits a finite triangulation $T$ with
	geodesic edges and that $F|_{\sigma}$ is a diffeomorphism when restricted to any (closed)
	simplex $\sigma\in T$, treating $\sigma$ as a manifold with boundary.
	
	Suppose also that on each simplex $\sigma$, $F|_{\sigma}^{*}h$ is close to $g$
	in the sense that for any $X\in T\sigma$,
	\begin{equation}
		\label{eq:integral-est-metric-bound}
		|X|_{F|_{\sigma}^{*}h}\leq C|X|_g
	\end{equation}
	for some constant $C$ which does not depend on $\sigma$.
	
	Then for any points $p,q\in M$,
	\[ d_h(F(p), F(q))\leq Cd_g(p,q). \]
\end{proposition}

\subsection{Circle packing results}
The following results are on circle packing in the sense of Rodin-Sullivan.

\begin{theorem}[Proposition 6.1 of \cite{circle_packing}]
	\label{thm:hyp-circle-unif}
	Let $K$ be a combinatorial closed disc (that is, simply connected, finite, and
	with nonempty boundary). Then there exists an essentially unique univalent
	circle packing $\mathcal{P}_K\subset \D$ for $K$ such that every boundary
	circle is a horocycle.
\end{theorem}

The following is proven in \cite{rodin-sullivan} (see also \cite{circle_packing}).
\begin{lemma}[Length-Area Lemma]
	\label{thm:length-area}
	
	Let $c$ be a circle in a circle packing in the unit disk. Let $S_1, S_2,\dots,
	S_k$ be $k$ disjoint chains which separate $c$ from the origin and from a
	point on the boundary of the disk. Denote the combinatorial lengths of these
	chains by $n_1,n_2,\dots, n_k$. Then
	\[ {\radius}(c)\leq \frac{1}{\sqrt{n_1^{-1}+n_2^{-1}+\dots+n_k^{-1}}}.\]
\end{lemma}

We state below the more general Lemma of He and Rodin \cite{he-rodin-bdd-valence} generalizing the Hexagonal Packing Lemma of Rodin-Sullivan \cite{rodin-sullivan} and He and Aharonov's \cite{He-hex-packing,aharonov2} estimates of the hexagonal packing constants.
\begin{theorem}[Theorem 2.2 of \cite{he-rodin-bdd-valence}] \label{thm:hexagonalpackinglemma}
Let $n$ be an integer $\geq 2$ and let $P_n$ be a circle packing in $\C$ such
that
\begin{enumerate}
	\item The valence of $P_n$ is bounded by $k_0$,
	\item The radii of the circles of $P_n$ are all bounded above by some positive
	$r$, and
	\item there is some ``center'' circle $c_0$ of $P_n$ such that the carrier of
	$P_n$ contains a closed disc of radius $(2n+1)r$ which is concentric with
	$c_0$.
\end{enumerate}

Let $P_n'$ be any other circle packing in $\C$ combinatorially equivalent to
$P_n$. Suppose that $c_0$ is surrounded by circles $c_1, c_2, \dots, c_k$ in
$P_n$ and that $c_0',c_1',\dots, c_k'$ are the corresponding circles in $P_n'$.
Let \[ d_1(P_n,P_n') := \max\left\{
\frac{\radius(c_j')/\radius(c_l')}
{\radius(c_j)/\radius(c_l)}\, :\, 0\leq j,l\leq k \right\}, \]
and let
\begin{equation}
	\label{eq:he-rodin-hex-packing}
	s(P_n):= \sup_{P_n'} \left( d_1(P_n,P_n')-1 \strut \right).
\end{equation}

	Then there is a constant $C$ depending only on $k_0$ such
	that $s(P_n)\leq C/n$.
\end{theorem}

We here state the circle packing versions of the Discrete Schwarz Lemma and the
upper bound on $H_n(v)$, proving the latter, before proceeding to the general
case.

\begin{theorem}[Circle Packing Schwarz Lemma, Thm 5.1 of \cite{rodinschwarz}]
	\label{thm:schwarz-lemma-circles}
	There is an absolute constant $\alpha$ with the following property. Let $HCP_m$ be
	$m$ generations of the regular hexagonal circle packing. Let $D$ be the
	smallest disk which contains $HCP_m$. Let $HCP_m'$ be any circle packing
	combinatorially equivalent to $HCP_m$ and also contained in $D$. Then
	\begin{equation}
		\label{eq:schwarz-lemma-circles}
		R_0'\leq \alpha R_0,
	\end{equation}
	where $R_0$ and $R_0'$ are the radii of the generation zero circles in $HCP_m$ and $HCP_m'$
	respectively.
\end{theorem}

\begin{corollary}[Theorem 6.2 of \cite{rodinschwarz}]
	\label{thm:circles-HK-bound}
	Let $K\subset \Omega$ be compact. There is a constant $H_K$ with the following
	property. Let $\epsilon >0$ be sufficiently small and let $c\mapsto c'$ be the
	circle packing isomorphism of an $\epsilon$-circle packing approximation
	$\Omega_{\epsilon}$ of $\Omega$ onto a suitably normalized circle packing
	$D_{\epsilon}$ of the unit disk $\D$. Then $R(c')/R(c)\leq H_K$ for all
	circles $c$ of $\Omega_{\epsilon}$ which intersect $K$.
\end{corollary}

\begin{proof}
	Consider a circle $c$ in $\Omega_{\epsilon}$ such that $c\cap K\neq
	\emptyset$. Let $m$ be maximal with respect to the property that
	$\Omega_{\epsilon}$ contains a copy of $HCP_m$ centered at $c$ and take
	$\Delta$ to be the smallest disk containing this $HCP_m$.
	
	The circle packing isomorphism $\Omega_{\epsilon}\to D_{\epsilon}$ gives us a
	corresponding $HCP_m'$ contained in the unit disk $\D$. Let $\lambda$ be the
	radius of the disk $\Delta$ and rescale the unit disk $\D$ so it has the same
	radius as $\Delta$. Applying Theorem \ref{thm:schwarz-lemma-circles} gives
	\[ \lambda R(c')\leq \alpha R(c) .\]
	
	Next we can assume $\epsilon$ is as small as we like, so in particular we can
	take $\epsilon < 1/4 d(K, \C\setminus \Omega)$. With this bound on $\epsilon$,
	$\lambda$ is bounded below by \[ \frac{1}{2}d(K,\C\setminus \Omega)\leq
	\lambda, \] for if not, then $\lambda < d(K, \C\setminus \Omega)/2$ and so
	\begin{align*}
		\lambda+2\epsilon
		& < \frac{1}{2}d(K,\C\setminus \Omega) +2\epsilon \\
		& < d(K,\C\setminus\Omega).
	\end{align*}
	But this is a contradiction because we assumed that $\Delta$ is the smallest
	disk containing $HCP_m$. Hence $\lambda\geq d(K,\C\setminus \Omega)/2$ and the
	ratio of radii becomes
	\[ \frac{R(c')}{R(c)}\leq \frac{\alpha}{\lambda}\leq
	\frac{2\alpha}{d(K,\C\setminus\Omega)} \] and we have the result, with $H_K:=
	2\alpha/d(K,\C\setminus\Omega)$.
\end{proof}

The equivalent of the circle packing discrete Schwarz Lemma would be the following conjecture.

\begin{condition}[Discrete Schwarz Lemma]
	\label{thm:disc-schwarz}
	There is a constant $\alpha$ independent of $v$ and $m$ with the following property.
	Let $D_m(v)$ be a realized closed combinatorial disk of generation $m$ and let
	$R>0$ be the radius of the smallest geodesic ball $B_R(v)$ containing
	$D_m(v)$. Let $\Phi$ be a barycentric discrete conformal map scaled so that
	the image $\Phi(D_m(v))$ is contained in a geodesic ball $B_R(\Phi(v))$ with
	the same radius $R$ and centered at $\Phi(v)$. Then
	\[ e^{\tilde{f}(v)}\leq \alpha e^{f(v)}\] where $\tilde{f}(v)$ and $f(v)$
	are discrete conformal factors at $v$ in $\Phi(D_m(v))$ and $D_m(v)$ respectively.
\end{condition}

Assuming the above condition is true, the following corollary is an easy consequence.

\begin{corollary}
	\label{thm:Hn-bdd}
	Let $K\subset \Omega$ be compact and choose $R$ such that $B_R(p)$ is a
	geodesic ball for every $p\in K$. Let $v$ be a vertex in $K$ and assume $m$ is
	maximal with respect to the property that there exists a realized closed
	combinatorial disk $D_m(v)$ of generation $m$ centered at $v$ with $D_m(v)$
	completely contained in $B_R(v)$. Then $H_n(v)\leq \alpha/R$, where $\alpha$ and $R$ are
	independent of $m$.
\end{corollary}

\begin{proof}
	Rescaling the image $\phi(D_m(v))$, we have by the Discrete Schwarz Lemma
	Condition that $Re^{{\tilde{f}_n}(v)}\leq \alpha e^{f_n(v)}$. Hence $H_n(v) =
	e^{{\tilde{f}_n}(v)}/e^{f_n(v)}\leq \alpha/R$, as required.
\end{proof}

The above corollary immediately implies that $H_n(v)$ is bounded above by a
constant $H_K$ depending only on the compact set $K$, which is what we need in
order for the third condition on an admissible sequence (Definition
\ref{def:admissible-sequence}) to hold.

\subsection{Vertex scaling results}

The following results are from the literature on vertex scaling discrete conformal structures.
Let $L$ be a lattice in the complex plane $\C$. Then there exists a Delaunay
triangulation $\mathcal{T}_{st}= \mathcal{T}_{st}(L)$ of $\C$ with vertex set
$L$ such that $\mathcal{T}_{st}$ is invariant under the translation action of
$L$.

Given the triangulation $\mathcal{T}_{st}$, we can define certain subcomplexes
$\mathcal{B}_m(v)$ as follows: Let $B_m(v)= \{i\in V(\mathcal{T}_{st})\, : \,
d_c(i,v)\leq m\}$, where the distance $d_c(i,v)$ is the combinatorial distance
between the vertices $i$ and $v$. Note that $B_m(v)$ is a set of vertices.
Let $\mathcal{B}_m(v)$ be the subcomplex of $\mathcal{T}_{st}$ whose simplices
have vertices contained in $B_m(v)$. 

We also need to define the phrase ``embeddable development map'', which
Luo-Sun-Wu do in the following two definitions:

\begin{definition}[Bottom of p 11 of \cite{luo-sun-wu_convergence}]
	If $(S,\mathcal{T},l)$ is a flat generalized PL metric on a simply connected
	surface $S$, then a \emph{developing map} $\phi: (S,\mathcal{T},l)\to \C$ for
	$(\mathcal{T},l)$ is an isometric immersion determined by $|\phi(v)-\phi(v')|
	= l(vv')$ for $v\sim v'$.
\end{definition}

\begin{definition}[Definition 4.1 in \cite{luo-sun-wu_convergence}]
	A flat generalized PL metric on a simply connected surface $(X,\mathcal{T},l)$
	with developing map $\phi$ is said to be \emph{embeddable} into $\C$ if for
	every simply connected finite subcomplex $P$ of $\mathcal{T}$, there exists a
	sequence of flat PL metrics on $P$ whose developing maps $\phi_n$ converge
	uniformly to $\phi|_P$ and $\phi_n: P\to \C $ is an embedding.
\end{definition}

Now we can state the vertex scaling version of the hexagonal packing lemma:

\begin{lemma}[Lemma 4.7 of \cite{luo-sun-wu_convergence}]
	\label{thm:VS-RS-LDCR}
	Take a standard hexagonal lattice $V=\Z + e^{2\pi/3}\Z$ and its associated
	standard hexagonal triangulation whose edge length function is $L: V\to
	\{1\}$. There is a sequence $s_n$ of positive numbers decreasing to zero with
	the following property. For any integer $n$ and vertex $v$, there exists
	$N=N(n,v)$ such that if $m\geq N$ and $(\mathcal{B}_m(v), w * L)$ is a flat
	Delaunay triangulated PL surface with embeddable developing map, then the
	ratio of the lengths of any to edges sharing a vertex in $\mathcal{B}_m(v)$ is
	at most $1+s_n$.
\end{lemma}

\begin{acknowledgement}
	The authors would like to thank Ulrike B\"ucking, Joel Hass, Feng Luo, Yanwen Luo, Ken Stephenson, and Tianqi Wu for helpful conversations related to this work. The authors were supported by NSF DMS 1760538, CCF 1740858, and DMS 1937229.
\end{acknowledgement}

\bibliographystyle{alphaurl}

\bibliography{biblio-2.bib}
\end{document}